\documentclass[preprint]{imsart}

\usepackage{amsthm,amsmath}
\RequirePackage{natbib}
\RequirePackage[colorlinks,citecolor=blue,urlcolor=blue]{hyperref}
\usepackage{amssymb}
\usepackage{accents,xcolor}
\usepackage{mathtools}
\usepackage{mleftright}
\usepackage{lscape} 
\usepackage{xfrac,enumerate}
\usepackage[all]{xy}

\startlocaldefs

\newcommand{\pr}{\mathbb{P}}
\newcommand{\R}{\mathbb{R}}

\newcommand{\N}{\mathcal{N}}
\newcommand{\D}{\mathcal{D}}

\DeclarePairedDelimiterX{\norm}[1]{\lVert}{\rVert}{#1}
\DeclarePairedDelimiterX{\abs}[1]{\lvert}{\rvert}{#1}
\DeclarePairedDelimiterX{\0norm}[1]{\lVert}{\rVert_{0}}{#1}
\DeclarePairedDelimiterX{\1norm}[1]{\lVert}{\rVert_{1}}{#1}
\DeclarePairedDelimiterX{\2norm}[1]{\lVert}{\rVert_{2}}{#1}
\DeclarePairedDelimiterX{\nnorm}[1]{\lVert}{\rVert_{n}}{#1}
\DeclarePairedDelimiterX{\2nnorm}[1]{\lVert}{\rVert_{n}^2}{#1}

\newcommand{\hf}{\hat{f}}

\newcommand{\e}{\epsilon}

\let\oldsqrt\sqrt
\def\sqrt{\mathpalette\DHLhksqrt}
\def\DHLhksqrt#1#2{%
\setbox0=\hbox{$#1\oldsqrt{#2\,}$}\dimen0=\ht0
\advance\dimen0-0.2\ht0
\setbox2=\hbox{\vrule height\ht0 depth -\dimen0}%
{\box0\lower0.4pt\box2}}

\newtheoremstyle{break}
  {\topsep}{\topsep}%
  {\itshape}{}%
  {\bfseries}{}%
  {\newline}{}%
\theoremstyle{break}
\newtheorem{definition}{Definition}[section]
\newtheorem{lemma}[definition]{Lemma}
\newtheorem{theorem}[definition]{Theorem}

\newtheorem{Coro}[definition]{Corollary}

\newtheorem{remark}{Remark}[section]

\endlocaldefs

\begin{document}
\begin{frontmatter}

\title{Prediction bounds for higher order total variation regularized least squares}
\runtitle{Higher order total variation}


\author{\fnms{Francesco} \snm{Ortelli}\corref{}\ead[label=e1]{fortelli@ethz.ch}}
\address{R\"{a}mistrasse 101\\ 8006 Z\"{u}rich \\ \printead{e1}}
\and
\author{\fnms{Sara} \snm{van de Geer}\ead[label=e2]{geer@ethz.ch}}
\address{R\"{a}mistrasse 101\\ 8006 Z\"{u}rich \\ \printead{e2}}

\affiliation{Seminar for Statistics, ETH Z\"{u}rich}

\runauthor{Ortelli, van de Geer}

\begin{abstract}
We establish adaptive results for trend filtering: least squares estimation 
with a penalty on the total variation of 
$(k-1)^{\rm th}$ order differences.
Our approach is based on combining a general  oracle inequality for
the $\ell_1$-penalized least squares estimator
with ``interpolating vectors" to upper-bound the ``effective sparsity". This allows one to show that  the $\ell_1$-penalty on the $k^{\text{th}}$ order differences leads to an estimator that can adapt to the number of jumps in the $(k-1)^{\text{th}}$ order differences of the underlying signal
or an approximation thereof. We show the result for $k \in \{1,2,3,4\}$ and indicate how
it could be derived for general $k\in \mathbb{N}$.
\end{abstract}



\end{frontmatter}
\section{Introduction}\label{dualsec01}

Total variation penalties have been introduced by \cite{rudi92} and \cite{stei06}.
The present paper builds further on the theory as developed in \cite{tibs14}, \cite{wang16},  and \cite{gunt17}.
We show, for $k \in \{ 1,2,3, 4\}$, a method for proving that  the $k^{\text{th}}$ order total variation regularized least squares estimator adapts to the number of jumps in the $(k-1)^{\text{th}} $ order differences and indicate
how this method could be generalized to any $k \in {\mathbb N}$. 
Inspired by \cite{cand14}, our main tool is a  well-chosen vector interpolating the signs of the jumps.

The estimation method we will study is known as trend filtering. See \cite{Ryan2020} for a comprehensive overview
and connections. Trend filtering 
 is a special case of the Lasso (\cite{tibs96}): it is  least squares estimation with an $\ell_1$-penalty on
a subset of the coefficients. For trend filtering,
the minimization problem can also be formulated as a so-called analysis problem 
(\cite{elad07}) with the analysis matrix $D$ being
the $k^{\rm th}$ order differences operator (see Equation (\ref{Delta.equation})).
Our main result  problem, given in  Theorem \ref{dualthm101},
 is based on an oracle inequality for the general analysis problem with
 arbitrary analysis matrix $D\in \R^{m \times n}$, as given in Theorem \ref{dualthm45}.
 The latter is a modification of results in 
\cite{dala17}: we generalize their projection arguments by allowing for adding ``mock" variables to the active set.
We furthermore use an improved version of their  ``compatibility constant" 
(see Remark \ref{compatibility.remark}) and - up to  scaling - refer to its
 reciprocal as ``effective sparsity'', see Definition \ref{dualdef41}. The effective sparsity for the Lasso problem is 
 the number of active parameters (the sparsity) discounted by a factor due to correlations between variables.
 This discounting factor is called the compatibility constant (see Remark \ref{compatibility.remark}). 
 In our situation the effective sparsity  can be dealt with invoking
 what we call an ``interpolating
 vector"  (see Definition \ref{interpolating.definition}) which can be seen as a quantified noisy version of the so-called dual certificate used in basis pursuit. See Remark \ref{candes.remark} for more details. 

 Consider an $n$-dimensional Gaussian vector of independent observations $Y \sim{\cal N}_n (f^0, I)$ with unknown mean vector $f^0\in \R^n$, and with known variance ${\rm var}  (Y_i) = 1$, $i=1 , \ldots , n $
(see Remark \ref{compatibility.remark} for the case of
unknown variance). Let $D\in \R^{m \times n } $ be a given analysis matrix.
The analysis estimator is
\begin{equation}\label{analysis.equation}
\hat{f}:= \arg \min_{f \in \R^n} \biggl \{ \| Y - f \|_n^2  + 2 \lambda \| D f \|_1 \biggr \}, 
 \end{equation}
 where we invoke the (abuse of) notation $\| v \|_n^2:= \sum_{i=1}^n v_i^2 / n $, $v \in \R^n$. 
The general aim is to show that $\hat f$ is close  to the mean $f^0 :=\mathbb{E} Y$ of $Y$, or to some approximation $f \in \R^n$ thereof that has $\| D{f} \|_0$ ``small".

 The trend filtering problem has as analysis matrix $D$ the $k^{\text{th}}$ order differences
 operator $\Delta(k)\in  \R^{(n-k)\times n}$, which is defined as
\begin{equation}\label{Delta.equation}
\Delta(k)_{ij}:= \begin{dcases} (-1)^l\binom{k}{l}, & j=i-l, \ l \in [0: k],\  i \in \D,\\
0, & \text{else},
\end{dcases}
\end{equation}
where $\D =[k+1: n]$ and  $ k \in [1:n-1]$ is fixed. 
We alternatively call $\Delta(k) $ the $k^{\rm th}$ order discrete derivative operator.
Moreover, we apply the notation
$$ [a:b ] = \{ j \in {\mathbb N}: a \le j \le b \} , \ 0 \le a \le b < \infty . $$
Theorem \ref{dualthm45} below presents results for the general
analysis problem and we apply it in Theorem \ref{dualthm101} to the trend filtering problem. This application means 
that we need to introduce a ``dictionary"
as described in Subsection \ref{general-dictionary.section}, to bound the lengths of
the dictionary vectors, and finally calculate an interpolating vector to obtain a bound for the effective sparsity.
 We do the calculations for $k \in \{ 1,2,3,4\}$ and sketch the way to proceed for general $k \in {\mathbb N}$. 
 
 \subsection{Related work}\label{related.section}
Total variation regularization and trend filtering have been studied from different
angles in a variety of papers. The paper \cite{mamm97-2} studies numerical adaptivity and rates of convergence.
In \cite{kim2009ell_1} it is shown that
interior point methods work well for trend filtering.
The paper \cite{tibs14} clarifies connections with splines
and also has minimax rates.  In \cite{wang16} trend filtering on graphs is examined
and it has theoretical error bounds in terms  of
the $\ell_1$-norm $\| D f \|_1$. 
 The paper \cite{sadh17} contains theory for additive
models with trend filtering. In  \cite{sadhanala2017higher} trend filtering in higher dimensions
is studied and minimax rates are proved.
The paper \cite{chatterjee2019adaptive} proposes a recursive partitioning
scheme for higher dimensional trend filtering.
 Our work  is closely related to the paper \cite{gunt17} which  concerns the constrained problem as well
as the penalized problem. 
 Our results for the penalized problem with $k\in \{ 2,3,4 \} $ improve those in \cite{gunt17}.
As a special case we derive that under a ``minimum length condition" saying that the distances between jumps of the $(k-1)^{\rm th}$
 discrete derivative are all of the same order, and under an appropriate condition
 on the tuning parameter $\lambda$,  the prediction error of the penalized least squares estimator
 is of order $(s_0+1) \log (n / (s_0 +1)) \log n/n$ where $s_0$ is the number of jumps of
 $\Delta (k-1) f^0$  (see Corollary \ref{dualthm11}).
 This is an improvement
 of the result in \cite{gunt17} where the rate is shown to be $(s_0+1)^{2k}/ n$ for $k \ge 2$.
 In fact, we show a more general result where $f^0$ may be replaced by 
 a sparse approximation.
 For $k=1$ we show the result with a superfluous log-factor: it is known that in that case
 the rate of convergence for the prediction error is of order
 $(s_0 +1) \log(n/(s_0+1) )/n$, see \cite{gunt17} and its references. Our extra log-factor is is due to the use of projection
 arguments instead of more refined empirical process theory. In 
 \cite{vdG2020}
 it is shown that the log-factor for $k=1$ can be removed when invoking entropy arguments  instead of projections,
 while keeping the approach via interpolating vectors and effective sparsity.
 
 The approach with interpolating vectors is in our view quite natural and lets itself be extended to
 other problems. We discuss this briefly in the concluding section, Section \ref{dualsec12}. 

 \subsection{Organization of the paper}
In the next subsection, Subsection \ref{mainresult.section},  we present in Theorem
\ref{dualthm101} an adaptive result for trend filtering, where adaptivity means that the presented bound
for the prediction error can be smaller when $f^0$ can be well approximated by
a vector with fewer jumps in  its $(k-1)^{\rm th}$ discrete derivative.
Section \ref{general.section} presents in Theorem \ref{dualthm45} adaptive and non-adaptive bounds for the general
analysis problem which will be our starting point for proving Theorem \ref{dualthm101}.
We introduce effective sparsity and interpolating vectors in Definitions \ref{dualdef41} and
\ref{interpolating.definition}. 

Section \ref{application.section} applies the general result of Theorem \ref{dualthm45} to the case $D= \Delta(k)$. 
We then need to introduce a projected dictionary for trend filtering, which is done in 
Subsection \ref{dictionary.section}. With this we arrive at non-adaptive, almost minimax rates
in Theorem \ref{dualthm12}. In Subsection \ref{effective-sparsity.section} we construct interpolating vectors and bounds for
the effective sparsity for the case $k \in \{1, 2,3,4\}$ and also sketch how this can be
done for general $k$. With these results in hand we finish in Subsection \ref{proofmain.section} the proof
of the adaptive bounds for trend filtering with $k \in \{ 1,2,3,4 \}$. 
Section \ref{dualsec12} concludes the paper.  

The appendix contains a proof of Theorem \ref{dualthm45}. Its arguments are to a large extent in \cite{dala17} and 
\cite{orte19-2}, but there are modifications. The appendix also has the proofs for Subsection \ref{dictionary.section} and
\ref{effective-sparsity.section}.

  \subsection{Main result for trend filtering}\label{mainresult.section}
 
 For $D=\Delta (k)$ and $\D = [k+1 : n ]$
we let $S  =\{t_1, \ldots, t_s\} \subseteq \D, \ k+1 \le t_1 < \ldots < t_s \le n$ and let $t_0:= k$ and $t_{s+1}:= n+1$. We  define $n_i=t_i-t_{i-1}, \ i \in [1:s+1]$ and $n_{\rm max}:= \max_{1 \le i \le s+1} n_i $. 
Moreover, for $f \in \R^n$ we write $(\Delta(k) f)_{-S} := \{ ( \Delta (k) f)_j : \ j \in \D \backslash S \} $. 

In Theorem \ref{dualthm101} below, the set $S$ is fixed but arbitrary. The theorem presents an oracle inequality that
allows for a trade-off between approximation error and estimation error by choosing $S$ and $f$ appropriately,
depending on the unknown $f^0$. However, the tuning parameter will then depend on $s$. Remark
\ref{lambda.remark} reverses this viewpoint.

Write for $u>0$,
$$\lambda_0 (u):= \sqrt{2\log(2(n-k-s))+2u  \over n} . $$

\begin{theorem}[Adaptive rates for $k=1, 2, 3,4$]\label{dualthm101} 
Let $k \in \{1 , 2 , 3,4 \}$. There exists constants $c_k$ and $C_k$  depending only on $k$ such that the following holds. \\
Let $u>0$ be arbitrary and choose the tuning parameter $\lambda$ satisfying
$$ \lambda \ge c_k n^{k-1} \biggl ( { n_{\max} \over 2 n}\biggr ) ^{2k-1 \over 2}  \lambda_0(u).  $$
Let $f\in \R^n$ be arbitrary
and define the signs 
$$q_{t_i}: = {\rm sign} (Df)_{t_i} , \ i=[1 : s] . $$
Write
$S^{\pm} := \{ i \in [2: s] : q_{t_{i} } q_{t_{i-1}} = -1 \} \cup \{1 , s+1 \} $. 
Assume $n_ i\ge k(k+2) $ for all $i \in S^{\pm} $. Finally, let $v>0$ be arbitrary.
Then with probability at least $1-e^{-u}-e^{-v}$ we have
\begin{eqnarray*}
\norm{\hf-f^0}^2_n &\le& \underbrace{\norm{f-f^0}^2_n}_{\mbox{\it ``approximation error"}}  + 4 \lambda \norm{( \Delta(k)f)_{-S}} _1\\
&+&  \underbrace{ \left( \sqrt{\frac{k(s+1)}{n}} + \sqrt{ \frac{2v}{n}}+  \lambda \Gamma_S \right)^2}_{\mbox{\it ``estimation error"}}
\end{eqnarray*}
where
\begin{equation} \label{gammabound.equation}
 \Gamma_S^2  = n C_k \left (  \sum_{i \in S^{\pm}  } { 1 + \log n_i \over n_i^{2k-1}} +
\sum_{i \in S \backslash S^{\pm} }  { 1+ \log n_i \over n_{\rm max}^{2k-1}   } \right ) .
\end{equation} 
\end{theorem}

To prove this result, we will invoke Theorem \ref{dualthm45}. This requires providing a 
dictionary and bounding the effective sparsity
given by Definition \ref{dualdef41}. In Subsection \ref{proofmain.section} we then put the pieces together.

%

\begin{remark} The quantity $\Gamma_S^2$ in the above theorem is a bound for the
effective sparsity.
\end{remark} 

\begin{remark} One may take $c_1=c_2=2$.  For $\min_{i \in S^{\pm}}  n_i \rightarrow \infty$
asymptotic expressions for  $c_3$ and $c_4$ can be taken to be
$c_3 \rightarrow 19/2$ and by numerical computation, $c_4 \rightarrow 2\times 6^{7/2} / (18.62)\approx 56.83$.
 See Subsection
\ref{effective-sparsityk.section}. 
\end{remark}

\begin{remark}\label{lambda.remark}
The requirement for $\lambda$ depends on $S$. Therefore, given a fixed data-independent $\lambda$, the above theorem holds for the restricted selection of active sets $S$ having $n_{\max}$ upper bounded as 
$$ n_{\max} \le \left(\frac{\sqrt n\lambda}{c_k \lambda_0 (u) } \right)^{\frac{2}{2k-1}}. $$
\end{remark}

\begin{remark}
In Section \ref{effective-sparsityk.section}, we  indicate how the bound of Theorem \ref{dualthm101} 
could be established for general $k\in {\mathbb N}$.
\end{remark}

We formulate a corollary for the case where 
the distances between jumps are all of the same order as the maximal distance $n_{\max}$.
To facilitate the statement we give an asymptotic formulation. 
For sequences $\{ a_n \}$ and $\{ b_n\} $ in $(0, \infty) $  we use the notation $a_n = {\mathcal O} (b_n)$
if $\limsup_{n \rightarrow \infty} a_n / b_n < \infty $ and $a_n \asymp b_n$ if also
$b_n / a_n = {\mathcal O} (1)$. For a sequence of random variables $\{ Z_n \}$ we write
$Z_n = {\mathcal O}_{\mathbb P } (1) $ if $\lim_{M \rightarrow \infty} \limsup_{n \rightarrow \infty} {\mathbb P}  ( | Z_n | > M ) =0 $.

\begin{Coro}\label{dualthm11}
Fix $k \in \{1,2,3,4\}$. Choose $S$ such that
$$ \min_{i \in [1: s+1]}  n_i \asymp n_{\max} . $$
Then we can choose $\lambda$ of order
$$\lambda\asymp n^{k-1} \biggl ( { 1 \over s+1} \biggr )^{2k-1 \over 2} \sqrt { \log n \over n} $$
and with this choice, for all $f \in \R^n$, 
\begin{eqnarray*}
 \norm{\hat{f}-f^0}^2_n &\le&  \norm{f-f^0}^2_n + 4 \lambda \norm{(\Delta(k)f)_{-S}}_1 \\
 &+&  {\mathcal O}_{\mathbb P}  \left( \frac{s+1} {n}  \log (n/ (s+1)) \log n\right).
 \end{eqnarray*}
\end{Coro}

\begin{remark} The requirement of Theorem \ref{dualthm101} on the tuning parameter is more flexible than
the one in  \cite{gunt17}. For example, in the context of the  above corollary, 
\cite{gunt17} require $\lambda/ n^{k-1} \asymp \sqrt {\log (n / (s+1) ) / n}  $.
If we use this latter choice of the tuning parameter we find from Theorem \ref{dualthm101}
\begin{eqnarray*}
 \norm{\hat{f}-f^0}^2_n &\le&  \norm{f-f^0}^2_n + 4 \lambda \norm{(\Delta(k)f)_{-S}}_1 \\
 &+&  {\mathcal O}_{\mathbb P}  \left( \frac{(s+1)^{2k}} {n}  \log (n/ (s+1)) \log n\right).
 \end{eqnarray*} 
 Thus, up to log-terms Theorem \ref{dualthm101} recovers the result of \cite{gunt17}, with $f = f^0$.  We however
 obtain an  improvement because we allow for a smaller tuning parameter.
 Moreover, we allow for a sparse approximation $f$ of $f^0$ and present a sharp oracle inequality.
  The paper  \cite{gunt17}, also studies the constrained problem where they arrive at rates which are
 up to log-terms comparable
 to ours for the regularized problem when taking $f= f^0$.  
\end{remark} 

\begin{remark}\label{sigma.remark}
Our method of proof is along the lines of \cite{dala17}.
 In \cite{orte19-2} it is shown that one can also use this method for the square-root Lasso.
 This means that as a corollary of \cite{orte19-2} our result also hold for  ``square-root" trend filtering, with a ``universal" choice for the tuning parameter
 in the sense that it does not depend on the variance of the noise. 
 \end{remark}


\section{Adaptive bounds for the general analysis estimator}\label{general.section} 

Recall the analysis problem 
\begin{equation}
\hat{f}:= \arg \min_{f \in \R^n} \biggl \{ \| Y - f \|_n^2  + 2 \lambda \| D f \|_1 \biggr \}
 \end{equation}
where $D \in \R^{m \times n}$ is a given analysis operator, $\lambda >0 $ is
a tuning parameter and $\norm{v}^2_n:= \norm{v}^2_2/n$, $v \in \R^n$.

%
%
%

To be able to state Theorem \ref{dualthm45} -
a modification of results in \cite{dala17} (see Remark \ref{compatibility.remark}) -
we introduce some notation in Subsection \ref{benchmark.section}, and then describe the dictionary 
(Subsection \ref{general-dictionary.section}) and the
effective sparsity (Subsection \ref{general-effective-sparsity.section}). 
Theorem \ref{dualthm45} can then be found in Subsection \ref{general-main.section}.
To apply it one needs to upper-bound the effective sparsity.
This is done in Lemma \ref{interpolating.lemma}, which invokes
interpolating vectors as defined in Definition \ref{interpolating.definition}
of Subsection \ref{interpolating.section}. 
Theorem \ref{dualthm45} and  Lemma \ref{interpolating.lemma} combined serve as starting point for proving the result for
trend filtering in Theorem \ref{dualthm101}.

\subsection{Some notation} \label{benchmark.section}
The rows of the  analysis matrix $D \in \R^{m \times n}$ are indexed by a set $\D$ of size $\abs{\D}= m$. 
We consider a set $S \subseteq \D$, which is arbitrary and can be chosen as the active set  of an ``oracle" that trades off ``approximation error" and 
``estimation error" (see Theorem \ref{dualthm45}). The size of $S$ is denoted by $s:= |S|$.  We define for a vector $b_{\D} \in \R^m$ with index set $\D$
$$ b_S := \{ b_j \}_{j \in S } , \ b_{-S} := \{ b_j \}_{j  \in \D \backslash S} . $$
We let $ \N_{-S} : = \{ f \in \R^n :\ (D f)_{-S}  =0 \} = \{ f\in \R^n : \ (Df)_j =0 \ \forall \ j \in \D  \backslash S \} $ and write
$r_S := {\rm dim} ({\cal N}_{-S} )$. As benchmark for our result, consider the active set $S_0 :=
\{ j:\ (Df^0)_j \not= 0 \} $. If $S_0$ were known, the least squares estimator
$$ \hat f_{\rm LSE} := \arg \min_{f \in  \N_{-S_0} } \| Y - f \|_n^2 $$
would satisfy: for all $v >0$, with probability at least $1- \exp[-v]$
$$\| \hat f_{\rm LSE} - f^0 \|_n \le \sqrt {r_{S_0} \over n} + \sqrt {2v \over n} .$$
This follows from a concentration bound for chi-squared random variables, see  
Lemma 1 in \cite{laur00}.
An aim is to show that the estimator $\hat f$ converges with the same rate $\sqrt {r_{S_0} / n }$,
modulo log-factors. In fact we aim at showing this type of result with $f^0$ potentially replaced by a sparse approximation.
We hope to be able to choose the active set $S$ of a sparse approximation such that $r_S$ is small. On the other hand,
as we will see, the distance of the ``non-active" variables
to the linear space ${\cal N}_{-S}$ will play an important role: the smaller this distance, the
less noise is left to be overruled by the penalty. Therefore, we allow for the possibility to
extend $\N_{-S}$ to a larger linear space $\bar {\cal N}_{-S}\supseteq \N_{-S}$. 
This can be done for instance by adding some ``mock" active variables to the active set.
We let $\bar r_S= {\rm dim} (\bar \N_{-S} )$.
Thus, $\bar r_{S} \ge r_S$ but in our application to trend filtering we will choose them of the same order.

\subsection{The dictionary}\label{general-dictionary.section}
Given the linear space $\bar \N_{-S}\supseteq \N_{-S}$ we can decompose a vector $f \in \R^n$ into its projection
$f_{{\bar \N}_{-S}}$ onto $\bar \N_{-S}$ and its projection onto the ortho-complement 
$\bar \N_{-S}^{\perp}$, which we call its anti-projection:
$$ f = f_{{\bar \N}_{-S}} + f_{{\bar \N}_{-S}^{\perp}}. $$
The anti-projection  is the part we want to overrule by the penalty. For this purpose, we define a dictionary
$\Psi^{-S} := \{ \psi_j^{-S} \}_{j \in \D \backslash S} \in \R^{n  \times (m-s) }$ such that for all $f\in \R^n $ and for $b_{-S} = (D f)_{-S} $
$$ f_{{\bar \N}_{-S}^{\perp}}= \Psi^{-S} b_{-S} . $$
In general there can be several choices for $\Psi^{-S}$. In the application to trend filtering that we consider in this paper, $\Psi^{-S}$
will be uniquely defined (when $\bar \N_{S}=\N_{-S}$ it holds that
$ \Psi^{-S}= D_{-S}^{\prime} (D_{-S} D_{-S}^{\prime}  )^{-1} $ with $D_{-S}$  being the matrix
$D$ with the rows indexed by $S$ removed). 

\subsection{Effective sparsity}\label{general-effective-sparsity.section}

The effective sparsity will be invoked to establish adaptive bounds. 

Fix some $u>0$. Its value will occur in the confidence level of the inequalities in Theorem \ref{dualthm45}.
We define
\begin{equation}\label{lambda0.equation} 
 \lambda_0 (u) := \sqrt{2 \log (2(m-s))+ 2u)  \over n} .
 \end{equation}
In what follows, we assume that the tuning parameter $\lambda$ satisfies
\begin{eqnarray}\label{lambda.equation} 
 \lambda \ge  \max_{j \in \D \backslash S } \| \psi_j^{-S}  \|_n \lambda_0 (u) . 
 \end{eqnarray}
 
 For a vector $w_{-S}$ with $0 \le w_j \le 1 $ for all $j \in \D \backslash S $ we write
 $$ (1- w_{-S}) ( Df)_{-S} := \{ (1- w_j ) ( Df)_j \}_{j \in \D \backslash S } . $$

\begin{definition}[Effective sparsity]\label{dualdef41}
%
Let $q_S\in \{-1,+1\}^{s}$ be a sign vector.  The noiseless effective sparsity is 
$$\Gamma^2(q_S ):= \left ( \max   \biggl \{ q_S^{\prime} (Df)_S - \| (D f)_{-S} \|_1 : \ \| f \|_n=1   \biggr \}\right )^{2}  . $$
The noisy effective sparsity is defined as 
$$\Gamma^2(q_S, w_{-S} ):=\left( \max  \biggl \{ \ q_S^{\prime} (Df)_S - \| (1- w_{-S} (D f)_{-S} \|_1 : \ \| f \|_n =1   \biggr \}\right )^{2}  ,$$ 
where
$$ w_j = \| \psi_j^{-S} \|_n\lambda_0 (u)  /\lambda, \   j \in \D \backslash S ,$$
with $\lambda$ satisfying (\ref{lambda.equation}). 
\end{definition}

\begin{remark} On the slightly negative side, when applying Theorem \ref{dualthm45} one may need to choose $\lambda$ strictly larger than (but of the same
order as)  
required in (\ref{lambda.equation}) in order to have a ``well-behaved"  effective sparsity. On the positive side, depending on
the situation, one may improve upon $\lambda_0 (u)$ in (\ref{lambda0.equation}) using bounds
for weighted empirical processes.
\end{remark}

\subsection{Main result for the general analysis problem}\label{general-main.section}
Recall that the set $S$ is arbitrary. In the following theorem, the set $S$ and also its vector
$f\in \R^n$ can be chosen to optimize the bounds by trading off approximation error
and estimation error. The theorem  provides adaptive bounds (oracle inequalities)
since the trade-off depends on the unknown signal $f^0$. 

\begin{theorem}\label{dualthm45}
Let $u >0$, $v>0$ and let the tuning parameter 
$\lambda $ satisfy (\ref{lambda.equation}).  Then $\forall f \in \R^n$ the following bounds hold: 
\begin{itemize}
\item
a non-adaptive bound: with probability at least $1-e^{-u}-e^{-v}$,
\begin{eqnarray*}
\norm{\hf-f^0}^2_n &\le& \norm{f-f^0}^2_n + 4 \lambda \norm{D f }_1\\
&+&  \left( \sqrt{\frac{\bar r_S}{n}} + \sqrt{ \frac{2v}{n}} \right)^2.
\end{eqnarray*}
\item
and an adaptive bound: with probability at least $1-e^{-u}-e^{-v}$,
\begin{eqnarray*}
\norm{\hf-f^0}^2_n &\le& \norm{f-f^0}^2_n + 4 \lambda \norm{(D f)_{-S} }_1\\
&+&  \left( \sqrt{\frac{\bar r_S}{n}} + \sqrt{ \frac{2v}{n}}+ \lambda\Gamma(q_S,w_{-S}) \right)^2,
\end{eqnarray*}
where $q_S= \text{sign}(Df)_S $.
\end{itemize} 
\end{theorem}

\begin{proof}[Proof of Theorem \ref{dualthm45}]
See Appendix \ref{Thm2.2.proof}.
\end{proof}

\begin{remark} \label{compatibility.remark} 
Theorem \ref{dualthm45} is a modification of
the findings  by \cite{dala17} who study the Lasso problem.
Theorem \ref{dualthm45} is in terms of the analysis problem, as in \cite{orte19-2}.
We furthermore allow for
augmentation of $\N_{-S}$. Moreover, \cite{dala17} and \cite{orte19-2} replace $\Gamma (q_S, w_{-S})$
by the larger quantity $ \max \{  \ \| (Df )_S \|_1 - \| (1- w_{-S} ) ( Df)_{-S} \|_1 : \| f \|_n =1  \}  = :
\sqrt {r_S} / \kappa (S, w_{-S})  $ where $\kappa^2 (S , w_{-S})$ is called the ``compatibility constant".
The paper \cite{orte19-2} derives oracle results for the square-root analysis problem, which is
the analysis version of the square-root Lasso introduced by \cite{bell11}. Joining these findings
allows to derive a square-root version of Theorem \ref{dualthm45} which can be applied
to the case of unknown noise variance (see also Remark \ref{sigma.remark}). 
\end{remark}

\subsection{Interpolating vectors} \label{interpolating.section}
To upper-bound the effective sparsity one may invoke interpolating vectors. 

\begin{definition}[Interpolating vector]\label{interpolating.definition} 
Let $q_S\in \{-1,+1\}^{s}$ be a sign vector.  We call the completed vector $q \in \R^m$ with index set $\D$ an interpolating vector
(that interpolates the given signs at $S$). 
\end{definition}

\begin{lemma}\label{interpolating.lemma}
Let $q_S\in \{-1,+1\}^{s}$ be a sign vector.
The noiseless effective sparsity satisfies
$$\Gamma^2(q_S )\le \inf \{ n \| D^{\prime} q \|_2^2 : \ q \ \mbox{interpolating}, \ 
| q_j | \le 1 \ \forall \ j \in \D \backslash S \} .$$
The noisy effective sparsity $\Gamma(q_S, w_{-S})$ satisfies
$$\Gamma^2(q_S, w_{-S} )\le \inf \{ n \| D^{\prime} q \|_2^2 :  \ q \ \mbox{interpolating}, \ 
| q_j | \le 1- w_j \ \forall \ j \in \D \backslash S \} ,$$
where
$$ w_j = \| \psi_j^{-S} \|_n\lambda_0 (u)  /\lambda, \   j \in \D \backslash S , $$
with $\lambda$ satisfying (\ref{lambda.equation}). 
\end{lemma}

\begin{proof}[Proof of Lemma \ref{interpolating.lemma}]
We only prove the statement of the lemma for the noisy case as the argument is the same for the noiseless case.
For any vector $q_{-S}$ with $|q_j | \le 1- w_j$ for all $j \in \D \backslash S$, it is true that for all $f$,
$$ \| (1- w_{-S} ) (Df)_{-S} \|_1 \ge q_{-S} (Df)_{-S} .$$ Therefore, for all $f$, 
\begin{eqnarray*}   q_S^{\prime} (Df)_S -  \| (1-w_{-S} )( D f)_{-S} \|_1
 \le  q^{\prime} Df \le  \sqrt{n} \norm{D'q}_2 \norm{f}_n.
\end{eqnarray*}
\end{proof}

\begin{remark}\label{candes.remark}
To bound the effective sparsity invoking interpolating vectors, as the above lemma does,
we were inspired by the
dual certificates as applied in \cite{cand14}. These have also been used in other works. The paper 
\cite{cand14} considers the
superresolution problem and develops an interpolating function to establish exact
recovery in the noiseless problem. The requirement on this interpolating function is that
it is in the range of the transpose of the design matrix.
Dual certificates can be found in earlier work as well, see for example \cite{candes2010probabilistic}.
The latter applies  a ``near" dual certificate to deal
with noisy measurements. However, their ``nearness" appears to be restricted to a special setting.
In \cite{tang2014near} the approach is related to ours
but very much tied down to the noisy superresolution problem.
We are not aware of any work where
an explicit connection is made between dual certificates and compatibility constants, the latter
being related to the reciprocal of effective sparsity: 
see Remark \ref{compatibility.remark}.
The relation between dual certificates and interpolating vectors on the one hand and compatibility 
on the other hand, appears to have been hidden in the literature. Moreover, the notion of compatibility
has developed itself over the years. In for example \cite{boyer2017adapting} 
it is shown that compatibility conditions do not hold for superresolution,
but they show it for an older version of compatibility, not for the newer version based on $\kappa^2 (S, w_{-S})$.
\end{remark}

\section{Application of Theorem \ref{dualthm45} when $D = \Delta(k)$}\label{application.section}
In order to apply Theorem \ref{dualthm45} with
$D = \Delta(k)$ we need to establish a bound for the length of the columns of an appropriate dictionary $\Psi^{-S}$.
This is done in Subsection \ref{dictionary.section}.
Then we can apply the first part of Theorem \ref{dualthm45}
and this will, as we will see in Subsection \ref{minimax.section}, result in the
minimax rate up to log-terms.
Next, for the application of the second part of Theorem \ref{dualthm45} to obtain
adaptive results, we upper-bound the effective sparsity using an appropriate interpolating vector.
This is done in Subsection \ref{effective-sparsity.section}. We then have all the material
to establish Theorem \ref{dualthm101} as is summarized in Subsection \ref{proofmain.section}.

\subsection{The dictionary when $D = \Delta (k)$} \label{dictionary.section}
We start with some remarks, whose purpose is mainly to introduce some further notation.
Note that by definition, ${\cal N}_{\D} := \{ f \in \R^n :\ (Df )_j =0 \ \forall \ j \in \D \} $. For vectors $f_{{\cal N}_{\D}}^{\perp} $ a dictionary  is denoted by $\Psi^{\D}$.
 If $D$ has full row rank (which will be the
case for $D=\Delta (k)$, see \cite{wang2014falling}) it holds that
 $\Psi^{\cal D} = D^{\prime} (D D^{\prime} )^{-1} $. This is the Moore Penrose pseudo inverse $D^+$ of $D$.
 Let now $\Psi = \{ \psi_j \}_{j=1}^n \in \R^{n \times n}$ be a ``complete" dictionary, which means that we can write each $f \in \R^n$ as
 $f = \Psi b $, where $b_{\cal D} = Df$. Then obviously $\Psi^{\D} = \{ ( \psi_j )_{{\cal N}_{\D}^{\perp}} \}_{j \in {\cal D} } $ is formed by the projections
 of the dictionary vectors $\psi_j$ with index in ${\cal D}$ on the ortho-complement of the space spanned by $\{ \psi_{j^{\prime}} \}_{j^{\prime} \notin \D } $.  
 Moreover the dictionary $\Psi^{-S}= \{ \psi_j^{-S} \}_{j \in \D \backslash S} $ for vectors $f_{{\cal N}_{-S}}^{\perp} $  has $\psi_j^{-S} = ( \psi_j )_{{\cal N}_{-S}^{\perp}}=( \psi_j^{\D}  )_{{\cal N}_{-S}^{\perp}}$.

 As said, we may want to augment the space ${\N}_{-S}$ to a larger linear space $\bar \N_{-S}$ so that the anti-projections
 will have smaller length. This is done by taking the direct product of ${\cal N}_{-S}$ with a space spanned by
 additional linearly independent vectors $\{ \phi_j \} $ with $\phi_j \notin {\cal N}_{-S}$ for all $j$. 
 We call these additional vectors ``mock" variables. One may pick them in the set $\{\psi_j \}_{j \in \D \backslash S }$ (or $ \{ \psi_j^{\D} \}_{j \in \D \backslash S } $) in which case we call them
 ``mock" active variables.  
 
 We now first present  upper bounds for $\{ \| \psi_j \|_2^2 \}_{j \in {\cal D} \backslash S }$
 for the case $D= \Delta(k)$ for general $k$. We then give exact expressions  for $k\in \{1,2,3\}$ to illustrate that
 the bounds are sharp.
 
 \subsubsection{The dictionary for general $k$: upper bounds}\label{dictionaryk.section}
We take $\bar \N_{-S}$ as the direct product of $\N_{-S}$ and
the space spanned by $\{ \psi_{t_i+1} , \ldots \psi_{t_i+k-1 } \}_{i=1}^s$ (assuming
$t_s+k-1 \le n $). In this way we disconnect the system into
 $s+1$ components each having the same structure. The matrix $D$ with the rows
 indexed by $S \cup \{ t_i+1 , \ldots ,  t_i+k-1  \}_{i=1}^s$ removed is a block matrix. 
  To avoid digressions more details are given below
 only for $k \in \{1,2,3\}$ but it is clear that such a disconnection works for general $k$.
The space $\bar \N_{-S}$ has dimension $\bar r_S = r_S + s(k-1) = k(s+1)$. One may apply
a reformulation for the sub-intervals $\{[t_{i-1}+1: t_i-1 ]\}_{i=1}^{s+1} $, to arrive at 
upper bounds for the lengths of the columns of $\Psi^{-S}$ from 
upper bounds for the lengths of the columns of $\Delta(k)^+$.

\begin{lemma}[An upper bound for the length of the columns of $\Delta(k)^+$.]\label{duallem39}
We have that for $j \in [k+1 : n-1]$
$$ \norm{\psi_j^{\cal D}}_2^2 \le  \min \biggl ( (j-k)^{2k-1}, (n+1-j)^{2k-1}\biggr )  . 
$$
\end{lemma}

\begin{proof}[Proof of Lemma \ref{duallem39}]
See Appendix \ref{dictionaryproofs.section}.
\end{proof}

It follows that 
\begin{equation}\label{lengthpsi.equation}
 \| \psi_{t_{i-1} +j }^{-S}  \|_2^2 \le \min \biggl ( j^{2k-1}, (n_i - j )^{2k-1} \biggr ) , \ j \in [1 : n_i-1] , \ i=[1 : s+1 ] . 
\end{equation}
Recall our abuse of notation $\| \cdot \|_n^2 = \| \cdot \|_2^2 / n$. We get
$$\max_{j \in [1:n_i-1] } \| \psi_{t_{i-1} +j}^{-S}  \|_n^2 \le { (n_i/2)^{2k-1} \over n}  ,  \ i=[1 : s+1 ].$$

We conclude that the requirement (\ref{lambda.equation}) on the tuning parameter
is met when
\begin{eqnarray}\label{lambda1.equation}
\lambda \ge  n^{k-1} \biggl ( {  n_{\rm max} \over  2 n  }  \biggr )^{2k-1 \over 2} \sqrt { 2 \log (2(n- s-k)) + 2u   \over n} .
\end{eqnarray}

\subsubsection{The dictionary when $k=1$: exact expressions}\label{dictionary1.section} The case $k=1$ has been well studied, see \cite{gunt17} and its references.
We include it here and later in Subsection \ref{effective-sparsity1.section} to highlight the additional argument needed when $k>1$.
We choose $\bar \N_{-S}= \N_{-S}$: no augmentation of this null space.
In the present case $D$ is the incidence matrix of a path graph:
$$D= \Delta (1) := \begin{pmatrix} 1 & -1 & 0  &\cdots & 0 & 0  \cr 
0& 1 & -1 & \cdots & 0 & 0\cr 
0 & 0 & 1 & \cdots & 0& 0  \cr 
\vdots & \vdots & \vdots & \ddots & \vdots& \vdots   \cr
0 & 0 & 0 & \cdots &1 & -1 \cr  \end{pmatrix} .$$
Thus by removing the rows in $S$ we see that we turn the matrix into a block matrix consisting of $s+1$ row-wise orthogonal  submatrices
that have, apart from columns consisting of only zero's, the same structure as the original matrix $D$. 
It means that the dictionary $\Psi^{-S}$ is of the same form as the dictionary
$\Psi^{\cal D} $ for the vectors $f$ in deviation from their mean. It is easy to see that this
dictionary is
$$ \psi_{i,j}^{\D}= {j -1\over n} {\rm l} \{ j \le i \} - {n-j+1 \over n} {\rm l}  \{ j > i \} , i \in [1:n] , \ j \in [2: n] .$$
The matrix $\Psi^{\cal D} \in \R^{n \times (n-1) } $ has squared column lengths
$$ \| \psi_j^{\D } \|_2^2 = { ( j-1) (n-j+1) \over n} , \ j \in [2:n ] . $$
It follows that for $i=1 , \ldots , s+1 $ and $j  \in [ 1 : n_i -1] $,
$$ \| \psi_{t_{i-1} +j}^{-S}  \|_2^2   = { j (n_i -j ) \over  n_i  }  . $$

\subsubsection{The dictionary when $k=2$: exact expressions}\label{dictionary2.section}
When the analysis operator $D$ is equal to $\Delta(2)$ its expression is
$$D = \Delta(2) = \begin{pmatrix} 1 & -2 & 1 & 0 & \cdots & 0 & 0 & 0\cr 
0 & 1 & -2 & 1 & \cdots & 0 & 0 & 0 \cr
0 &0 & 1 & -2 & \cdots & 0 & 0 & 0  \cr 
\vdots & \vdots & \vdots & \vdots & \ddots & \vdots & \vdots  \cr 
0 & 0 & 0 & 0 & \cdots & 1& -2 & 1  \cr \end{pmatrix}. $$
Therefore by removing only the rows in $S$ the matrix we end up with is no longer
consisting of row-wise orthogonal sub-matrices. However, if we remove
two rows at the time, those in $S $ together with their immediate followers $ \{ t_i+1 \}_{i=1}^s $ (for $t_s+ 1 \le n $ as otherwise at that point there is nothing to remove), we again
obtain $s+1$ sub-matrices that are, apart from their columns consisting of
only zero's, of the same form as $D$. The dictionary $\Psi^{-S}$ for functions $f_{\N_{-S}^{\perp} } $ has thus the same form as the  dictionary $\Psi^{\D} $ for the functions 
$f_{\N_{\D}^{\perp} } $, which are functions $f$ in deviation from
their projections on their linear part. 
Therefore it is useful to augment $\N_{-S}$ by taking
its direct product with the linear space spanned by the mock active variables $\{ \psi_{t_i+1 } \}_{i=1}^s $ to form $\bar \N_{-S}$. Note that $\bar r_S= r_S + s = 2(s+1)$.
Recall that $\Psi^{\D}$ is the Moore Penrose pseudo inverse $D^+$, in this case $\Delta(2)^+$.
%
We present the exact column lengths of $\Delta(2)^+ $ here:
they show that the upper bounds of Lemma \ref{duallem39} are, up to
constants, sharp.

\begin{lemma}[Length of the columns of $\Delta(2)^+$.]\label{duallem51}
The length of the columns $\{\psi^{\D}_j\}_{j \in [n]\setminus [2]}$ of the Moore Penrose inverse $\Psi^{\D}= \Delta(2)^+$ is 
(for $j \in [3: n] $)
$$ \norm{\psi^{\D}_j}^2_2= \frac{(n-j+1)(n-j+2)(j-2)(j-1)(2j(n-j+3)-3(n+1))}{6n(n+1)(n-1)}. $$
\end{lemma}

\begin{proof}[Proof of Lemma \ref{duallem51}]
See Appendix \ref{dictionaryproofs.section}. 
\end{proof}

  \subsubsection{The dictionary when $k=3$: exact expressions}\label{dictionary3.section}
  When $k=3$ we form $\bar \N_{-S}$ by taking the direct product of the linear space $\N_{-S}$
  and the space spanned by $\{ \psi_{t_i+1}, \psi_{t_i+2} \}_{i=1}^s $ (assuming $t_s+2 \le n$). 
  This space has dimension
  $\bar r_S = r_S + 2s = 3(s+1)$. The dictionary $\Psi^{-S}$ has blocks of the same form as the Moore Penrose pseudo inverse
  $\psi^{\D}=\Delta (3)^+$.
We compute the exact length of the columns of $\Delta(3)^+$. 

\begin{lemma}[Length of the columns of $\Delta(3)^+$.]\label{duallem61}
The length of the columns $\{\psi_j^{-S}\}_{j \in [4:n]}$ of the Moore Penrose pseudo inverse $\Delta(3)^+$ is,
for $j \in [4:n]$, 
\begin{eqnarray*}
\norm{\psi_j^{\cal D}}^2_2&=& \frac{(j-3)(j-2)(j-1)(n+3-j)(n+2-j)(n+j-j)}{60 (n+2)(n+1)n (n-1)(n-2)}\times \\
&& \times \left(10 (n+1)(n+2)+3j(n+4-j)(j(n+4-j)-4n-5) \right).
\end{eqnarray*}
\end{lemma}

\begin{proof}[Proof of Lemma \ref{duallem61}]
See Appendix \ref{dictionaryproofs.section}.

\end{proof}

\subsection{An almost minimax rate} \label{minimax.section}

Although minimax rates are not the main theme of this paper, we present a result in this direction
because it comes almost for free.

\begin{theorem} \label{dualthm12} Let $u >0$, $v>0$ and let the tuning parameter 
$\lambda $ satisfy (\ref{lambda1.equation}).  Then it holds that $\forall f \in \R^n$,
with probability at least $1-e^{-u}-e^{-v}$,
\begin{eqnarray*}
\norm{\hf-f^0}^2_n &\le& \norm{f-f^0}^2_n + 4 \lambda \norm{D f }_1
+  \left( \sqrt{\frac{k(s+1)}{n}} + \sqrt{ \frac{2v}{n}} \right)^2.
\end{eqnarray*}
\end{theorem}

\begin{proof}[Proof of Theorem \ref{dualthm12}]
This follows from applying the first, i.e., the  non-adaptive result of Theorem \ref{dualthm45}, and invoking that
its requirement (\ref{lambda.equation}) on the tuning parameter $\lambda$ is met if we impose the
bound given in (\ref{lambda1.equation}). 
\end{proof}

\begin{Coro}\label{minimax.corollary}
Recall that requirement (\ref{lambda1.equation}) on $\lambda$ depends on $S$ and the choice of $S$ is free
in Theorem \ref{dualthm12}. We can take $S$ such that 
$\min_{i \in [1:s+1] } n_i  \asymp n_{\rm max}$
in which case $n_{\max} /n\asymp 1/s$.  Then we may take
$$\lambda \asymp n^{k-1} \biggl ( { 1 \over s } \biggr )^{2k-1 \over 2} \sqrt {\log n \over n} .$$
Now $s$ is still a free parameter.
Choosing $s$ by a trade off, i.e., in such a way that $\lambda/ n^{k-1}  \asymp   s/n $,
we get for $n^{k-1} \| \Delta (k) f \|_1  \le  1$ (say)
$$ \|\hat f - f^0 \|_n^2 \le \| f- f^0 \|_n^2 + {\mathcal O}_{\mathbb P} ( n^{-{2k \over 2k+1 }}\log^{1 \over 2k+1} n) . $$
For $f=f^0$ this corresponds, up to the log-factor, to the minimax rate over
$\{ f^0: \ n^{k-1} \| \Delta (k) f^0 \|_1 \le 1\}$ (\cite{donoho1998minimax}). 

\end{Coro}

\subsection{Interpolating vectors and effective sparsity and  for $D = \Delta (k)$ } \label{effective-sparsity.section}
Observe that for  $D= \Delta(k)$ it holds that
$(D^{\prime} q )_{2k+j} = ((-1)^k \Delta(k) q )_{2k+j}$  for $j \in [1 : n-2k]  $ (this is in the background of partial integration).
The above observation
leads in the noiseless case to taking $q$ as piecewise $k^{\rm th}$ degree polynomial interpolation.
To avoid boundary effects, one may use an interpolation including the points $t_0:= k$ and
$t_{s+1} := n+1$, with $q_{t_0}=q_{t_{s+1}} =0 $.
Moreover, still in the noiseless case, if there is no sign change 
(i.e. $q_{t_{i-1}} q_{t_i} =1$) one can simply take $q_{t_{i-1}+j}  = q_{t_i}$ for $j \in [1, n_i-1] $.

In the noisy case one can use a similar interpolation except near the active points in $S$, where we need to change  to powers
$(2k-1)/2$ instead of $k$. This  is due to  (\ref{lengthpsi.equation}) 
and the requirement $|q_j |\le 1- w_j$ for all $j \in \D$ as used in 
Definition \ref{interpolating.definition}. It has an effect on the constants 
involved in the interpolating vector and moreover, for $t_{i-1} +j$ near active the points the absolute $k^{\rm th}$
discrete derivative 
$| \Delta (k) q_{t_{i-1} +j}  | $  behaves like $1/\sqrt j$. It follows from the next lemma
that this leads to an additional logarithmic factor
as compared to the noiseless case.

\begin{lemma}\label{sqrtj.lemma}
Let for some $d \in {\mathbb N}$, $d \ge 2k$, 
$${\bf q}_j:= j^{2k-1 \over 2} , \ j= k , \ldots , d . $$
Then for some constant $\tilde C_k$
$$ \| \Delta (k) {\bf q} \|_2^2 \le \tilde C_k^2(1+  \log d) . $$
\end{lemma}

\begin{proof}[Proof of Lemma \ref{sqrtj.lemma}]
See Appendix \ref{sparsityproofs.section}.
\end{proof}
For each given $k$ the interpolating vector
we suggest below can be computed by solving a system of linear equations.
The missing element is that it is not clear whether the interpolation is monotone
between two active points. 
We verify the monotonicity for $k=\{1,2,3,4\}$
in Subsections \ref{effective-sparsity1.section}, \ref{effective-sparsity2.section}, \ref{effective-sparsity3.section} and \ref{effective-sparsity4.section} respectively. 
In other words, Subsection \ref{effective-sparsityk.section} presents the general idea, and
the four following subsections work out the details for $k \in\{1,2,3,4\} $. 

\subsubsection{Construction of an interpolating vector}\label{effective-sparsityk.section} 
Define 
$$ S^{\pm} := \{ i \in [2: s] : q_{t_i} q_{t_{i-1}}  = -1 \} \cup \{ 1, s+1 \}  $$
and let $t_0 =k$ and $t_{s+1} = n+1$.
We call $[t_0 : t_1 ]$ the left boundary interval and $[t_s : t_{s+1} ]$ the right boundary interval.
We assume, that 
\begin{equation}\label{minimum-length.equation}
 n_i \ge k(k+2) \ \forall \ i \in S^{\pm}  . 
 \end{equation}

 For $i \in S^{\pm}$ we split $[t_{i-1}, t_i ]$ into $k+2$ subintervals
 of equal (Lebesgue) size when $k$ is even, and into $k+1$ sub-intervals when $k$ is odd.
 We call these sub-intervals the local sub-intervals.
 By (\ref{minimum-length.equation}) we are ensured that each local sub-interval has at least 
 $k$ elements. 
We call the left (right) sub-interval  of $[t_{i-1} : t_i]$ the left (right) local boundary interval.
 The other sub-intervals of the split will be called the local interior intervals.
 We will define $q_j$ for each local sub-interval and join them by
 discrete derivatives matching, the latter having the following meaning. Let $p_1 (j)$ and $p_2 (j)$ be two functions
 of $j \in [k+1:n] $. We then say that their $(k-1)^{\rm th}$ order discrete derivatives match at the point $j_0 \in [2k+1 : n-k+1]$ 
 if $p_1 (j) = p_2(j)$ for $j \in [j_0 : j_0 + k-1 ] $.
 \begin{itemize}
 
 \item
{\it The continuous version of the interpolating  vector.}
A continuous interpolation ${\rm q} : [0, 1] \rightarrow [-1 , 1 ] $
    with ${\rm q} (0)=1 $ and ${\rm q} (1) = -1 $ can be constructed as follows. 
    We choose ${\rm q}$ anti-symmetric around $x= 1/2$, i.e. given
    ${\rm q} (x)$ for $x \in [0 , 1/2] $ we let ${\rm q} (x) := - {\rm q}(1-x)$ for $x \in [1/2 ,1] $. 
    We split $[0,1]$ into $N$ intervals of equal size where $N=k+2$ if $k$ is even, and $N=k+1$  if $k$ is odd. Call these sub-intervals $\{ [x_{l-1}, x_l]\}_{l=1}^N$ (thus $x_0=0$, $x_N=1$ and $x_{N/2}=1/2$). For $x\in [x_0 , x_1 ] $ we let
    $$ {\rm q} (x):= 1- {\rm a}_0  x^{2k-1 \over 2} $$
    where the constant ${\rm a}_0 > 0 $ is to be determined.
    For $x\in [x_{N-1} , x_N ]$ we then have
    $${\rm q} (x)= -1+{\rm a}_0  (1-x)^{2k-1 \over 2} .$$ For $x \in [x_{N/2 -1} , x_{N/2 +1} ]$ we let
    ${\rm q}(x)=   {\rm a}_L (1/2-x)^L + \cdots + {\rm a}_1 (1/2-x) $ be a linear combination of odd powers of $(1/2-x )$
    where $L= k-1$ if $k $ is even and $L=k$ if $k$ is odd. By the anti-symmnetry, it remains do define
    ${\rm q} (x)$ for $x \in [x_1 , x_{N/2-1}]$. 
    For $x \in [ x_{l-1} , x_l ] $ with
    $l \in \{ 2 , \ldots , N/2 -1 \} $ we let ${\rm  q}(x) := {\rm b}_{l,k} x^k + \cdots +  {\rm b}_{l,1} x + {\rm b}_{l,0}   $
    be a polynomial of degree $k$. We choose the coefficients $\{  \{ {\rm a}_j \} , \{ {\rm b}_{l,j}  \} \}  $
    by derivatives matching: solving 
  a linear system with $ k (k/2)$ equations 
 with $ k (k/2)$ unknowns when $k$ is even, and $k (k-1) / 2 $ equations with $k (k-1) / 2 $ unknowns
 when $k$ is odd. 
    The resulting function ${\rm q}: [0,1 ] \rightarrow \R$ is interpolating
  between $+1$ and $-1$  and it is continuous with $k-1$ continuous derivatives, 
  such that the $k^{\rm th}$ left derivative is piece-wise constant except on
  the left boundary interval where it behaves like $-1 / \sqrt x$ and the right boundary
  interval where it behaves like $  (-1)^{k} / \sqrt {1 - x} $. 
  For 
  a given $k$, the coefficients $\{ \{ {\rm a}_j \} , \{ {\rm b}_{l,j} \} \}  $ can be
  given explicitly and it can then be checked whether ${\rm q}$ is a
  decreasing function on the  interval $[0,1]$ (or stays away from $\pm 1 $).
  We did this for $k\in \{1,2,3,4\}$ below.  

 \item
 {\it The case of a sign  change.}
 If $q_{t_{i}} q_{ t_{i-1}}=-1$ we apply a discrete version of the continuous function
 ${\rm q}$ described above. One way to do this is using the map
 ${t_{i-1} +j} \mapsto  q_{t_{i-1}} {\rm q} (j/ n_i )$ for $j\in [1: n_i-1]$. Alternatively, one
 may apply discrete derivatives matching. 
 We take $q_{t_{i-1} +j} $ anti-symmetric around the midpoint of $[t_{i-1} : t_i ] $. 
 We choose $q_{t_{i-1}} q_{t_{i-1} +j } := 1- a_0 (j/ n_i )^{2k-1 \over 2} $ for $j$ in the left local boundary
 interval (and thus $q_{t_{i-1}} q_{t_i -j }= -1 + a_0(1- j/ n_i )^{2k-1 \over 2 } $ at the right local boundary  interval) where $a_0>0 $ depends on $k$ and $n_i$. 
 At the two local interior intervals around the midpoint of $[t_{i-1}, t_i ]$ we take
 $q_{t_{i-1} +j } $ a linear combination of odd powers $l \le k$ of $(j- n_i/2)$. 
 At the other interior intervals we let $q_{t_{i-1} +j } $ be a polynomial of degree $k$.
 Then we choose the coefficient $a_0$ and the coefficients of the polynomials by discrete derivatives
 matching.  
For   $
  \min_{i \in S^{\pm} } n_i \rightarrow \infty$ the coefficient $a_0$ 
  converges to its continuous counterpart  ${\rm a}_0$. We conclude
  $a_0 \asymp 1$. The same is true for the other coefficients in the interpolation.

 \item
{\it The boundary intervals.}
We set $q_{t_0} = 0$, and
$\{ q_{t_0 +j} \}_{j=1}^{n_1-1} $ an interpolating vector constructed as for the case of a sign change,
except that we now interpolate between $0$ and $ \pm 1 $
instead of between $1$ and $-1$.
A similar construction is made for the right boundary interval $ [ t_{s}: t_{s+1} ]$ where we set
$q_{t_{s+1} } =0 $.

\item
{\it The case of no sign change.}
When $q_{t_i} q_{t_{i-1} }=1$ we take
$$ q_{t_i} q_{t_{i-1}+j } := 1- \biggl ( {4  j (n_i-j) \over n_i n_{\max}}\biggr )^{2k-1 \over 2 } , \ j=[1 : n_i -1].  $$

\item
 {\it Joining the intervals $\{ [t_{i-1} : t_i ]\}_{i=1}^{s+1}$.}
 By the above construction, the first order differences for $j \in [t_{i-1} :t_{i-1} + k  ]$ 
 or $j\in [t_{i}-k : t_{i}  ]$
  are all of order $n_i^{-{(2k-1) / 2}}$ and in fact of order $n_{\rm max} ^{-{(2k-1) /2}}$
  if there is no sign change.
  This means we can glue the interpolations for the intervals
  $\{[ t_{i-1}, t_i ] \}_{i=1}^{s+1}$ together and have the $k^{\rm th}$ discrete derivatives matching up to
 a finite (depending on $k$) number  of terms of order $ n_i^{-{(2k-1 )/ 2} } $
 (or even $n_{\rm max} ^{-{(2k-1) / 2}}$).

 \item{\it The $k^{\rm th}$ derivative of $q$.}
  The following bound holds: 
  for some 
  constant $C_k$
  \begin{eqnarray*}
   n \| \Delta (k)^{\prime} q \|_2^2 \le   C_k n \biggl  [
 \sum_{i \in S^{\pm} } { 1+ \log n_i \over n_i^{2k+1} } 
+  \sum_{i \notin S^{\pm} } { 1 + \log n_i \over n_{\rm max}^{2k+1}  }   \biggr  ] . 
\end{eqnarray*}
 This follows from the construction of $q$ and from Lemma \ref{sqrtj.lemma}.

\item
{\it The requirement $q_j \le 1- w_j$, $j \in \D \backslash S$.} 
For $i \notin S^{\pm}$ we have $|q_{t_{i-1} +j}| \le 1-w_{t_{i-1}+j} $, $j\in [1 : n_i-1]$ by construction as long as 
$\lambda $ satisfies (\ref{lambda1.equation}).  To conclude the same for  
$i \in S^{\pm}$
we strengthen the requirement
(\ref{lambda1.equation}) to: for an appropriate constant $c_k\ge 1$ depending only on $k$
\begin{equation}\label{lambdak.equation}
\lambda \ge c_k n^{k-1} \biggl ({n_{\rm max} \over 2n }  \biggr )^{2k-1 \over 2}   
\sqrt{2 \log (2(n-s-k))+ 2u) \over n}. 
\end{equation}
It can be shown that if the interpolations for $i \in S^{\pm}$ are monotone, then  one can take
as $\min_{i \in S^{\pm}} n_i \rightarrow \infty$
\begin{equation}\label{ck.equation}
c_k  \rightarrow \begin{cases}  {  2({ k+2  } )^{2k-1 \over 2} / {\rm a}_0} , & k \ {\rm even}\cr
{ 2 ({ k+1  }  )^{2k-1 \over 2} / {\rm a}_0} , & k \ {\rm odd}\cr  \end{cases} . 
\end{equation} 
\end{itemize}
However,
it is not a priori clear to us that the interpolations
$i \in S^{\pm}$ are monotone. We check this for $k= \{1,2,3,4\}$ in the next 4 subsections.

\subsubsection{Interpolating vector and effective sparsity for $k=1$} \label{effective-sparsity1.section}

In the noiseless case and
when $k=	1$ we take a linear interpolation of $(q_1:= 0, q_S^{\prime} , q_{n+1}:=0)$.
At a sign change: $q_{t_{i-1} } q_{t_i} =-1$ we take a linear interpolation 
between plus and minus one over an interval
of length $n_i$. The slope in this interval will then be $2/n_i$, which gives a contribution
$4/n_i$ to the bound for the effective sparsity. Similar observations can be
made for the boundary interval $[t_0: t_1]$ where we face a  boundary effect
due to partial
integration
because $q_2 = 1/ n_1 \not=0 $. For the right boundary interval $[t_s: t_{s+1}]$ we see the same boundary effect.
So for $k=1$
$$ \Gamma^2 (q_S) \le {n \over n_1^2} + {n \over n_1}+  \sum_{q_{t_i}  q_{t_{i-1}}=-1} {4 n \over n_i }   +
{n \over n_{s+1} } +{n \over n_{s+1}^2 }  .$$
\begin{remark}
Because the interpolating vector $q$ can be chosen to be constant  between consecutive entries of $q_S$ having the same sign, a ``staircase pattern'' -  consecutive entries of $(Df)_S$ having the same sign - seems to favour prediction, while for $f=f^0$ it is known to negatively affect model consistency (\cite{qian16}).
\end{remark}

In the noisy case, and when $i \in S^{\pm}$ we use a scaled discrete variant of
$${\rm q}(x) := \begin{cases} +1- \sqrt { 2 x} ,& 0 \le x \le 1/2 \cr -1 + \sqrt {2 (1-x)} , & 1/2 \le x \le 1 \cr
\end{cases} . $$
When $q_{t_{i-1}}q_{t_i}=-1$ we let
\begin{equation}\label{change1.equation}
q_{t_{i-1} +j} q_{t_{i-1}}  := {\rm q} ( j /n_i ) =
 \begin{cases} +1- \sqrt { 2 j / n_i} , &  j  \in [1:  n_i /2]  \cr -1 + \sqrt { 2 (n_i - j) / n_i}  ,&  j \in [n_i/2:  n_i ]  \cr
\end{cases} . 
\end{equation}
At the boundary intervals we let $q_1:=q_{n+1}:=0$,
\begin{eqnarray}\label{boundaryleft1.equation}
 q_{1+j } q_{t_1}  :=  \begin{cases}  \sqrt {j   / (2 n_1)} ,&  j \in [1:     n_1/2 ] \cr
1-  \sqrt { (n_1-j )/ (2n_1) }, &  j \in [n_1/2:  n_1] \cr 
 \end{cases} , 
 \end{eqnarray}
 and
 \begin{eqnarray}\label{boundaryright1.equation} 
q_{t_s+j } q_{t_s} :=  \begin{cases} 1-  \sqrt {   j /(2 n_{s+1} ) }, &   j \in [1:  n_{s+1} /2 ]\cr 
\sqrt { (n_{s+1} - j ) / (2 n_{s+1 } ) } , & j \in [n_{s+1}/2 :n_{s+1} ]  \cr 
\end{cases}  .
\end{eqnarray}
If $q_{t_i}q_{t_{i-1}}=1$ we take for $j \in [ 1: n_i-1] $
\begin{equation}\label{nochange1.equation}
 q_{t_{i-1} +j} q_{t_{i-1}}:= 1-  \sqrt {4 j ( n_i -j ) / (n_i n_{\rm max}) } . 
\end{equation}
In other words, at locations $t_i$, with $i\in [2: s]$, where the signs do not change (i.e.\
$q_{t_i } = q_{t_{i-1} }$) one may choose ``less steep" 
interpolations.

With this choice for $q$ we get by straightforward
calculations: for $\lambda $ satisfying (\ref{lambdak.equation}) with $c_1=1$, and for a universal constant $C_1$
\begin{eqnarray*}
\Gamma^2 (q_{S} , w_{-S}) & \le & n \| \Delta (1)^{\prime} q \|_2^2 \\ &\le & C_1 n \biggl  [
{ 1 + \log n_1 \over n_1} + \sum_{q_{t_i}q_{t_{i-1}}=-1}  { 1+ \log n_i \over n_i } \\
&+ & \sum_{q_{t_i} q_{t_{i-1}} =1} { 1 + \log n_i \over n_{\rm max} } + {1 + \log n_{s+1} \over n_{s+1 }  }  \biggr  ] . 
\end{eqnarray*}

\subsubsection{Interpolating vector and effective sparsity for $k=2$} \label{effective-sparsity2.section}
The splitting scheme for the noisy case of Subsection \ref{effective-sparsityk.section} applied to $k=2$
gives as continuous interpolation
$${\rm q} (x) := \begin{cases} +1- {2 \over 5} (4 x)^{3/2} ,& 0 \le x \le {1 \over 4}  \cr 
{3 \over 5} 4 ({1 \over 2} - x ) , & {1 \over 4} \le x \le { 3 \over 4} \cr
-1 + {2 \over 5} 4^{3/2}(1-x)^{3/2} , & { 3 \over 4} \le x \le 1 \cr 
\end{cases}\ \  . $$
Alternatively, we may use a simpler solution, namely
$${\rm q}_{\rm alt} (x) := \begin{cases} +1- (2 x)^{3/2} ,& 0 \le x \le {1 \over 2}  \cr -1 + ( 2 ( 1- x))^{3/2}, & {1 \over 2}  \le x \le 1 \cr  \end{cases} . $$
Then ${\rm q}_{\rm alt}$ is decreasing, ${\rm q}_{\rm alt}$ and ${\rm q}_{\rm alt}^{\prime} $ are continuous and
$$ {\rm q}_{\rm alt}^{\prime \prime} (x) = \begin{cases} - {3 \over \sqrt {2 x}},  &0 < x <  {1 \over 2} \cr
+ {3 \over \sqrt {2(1-x)}  }, & {1 \over 2} \le  x < 1 \cr\end{cases}\ . $$
Consider now a sign change: $q_{t_i}q_{t_{i-1}}=-1$. We may assume without loss of generality that
$n_i$ is even. (If $n_i$ is odd we take $q_{t_{i-1} +1} =q_{t_{i-1} } $ - which is possible because $t_{i-1}+1$ is set to be a
``mock" active variable - and replace $n_i$ by $n_i-1$.)  
 For $j \in [ 1: n_i/2] $ define 
 \begin{equation}\label{change2.equation}
q_{t_{i-1} +j} q_{t_{i-1}}  := {\rm q}_{\rm alt} ( j /n_i ) =
 \begin{cases} +1- ( 2 j / n_i)^{3/2} , & j \in [ 1: n_i/2 ]  \cr -1 + (2 (n_i - j) / n_i)^ {3/2} ,  & j \in [n_i/2 : n_i ] \cr
\end{cases} . 
\end{equation}
At the boundary intervals, say the left boundary interval, we let $q_2=0$. Again we may without loss of generality assume
$n_1$ is even (otherwise we let $q_3=0$ and replace $n_1$ by $n_1-1$).  Then let
\begin{eqnarray}\label{boundary2.equation}
 q_{2+j} q_{t_1}  :=  \begin{cases}  \sqrt 2(j   /  n_1)^{3/2} , &  j  \in [1:   n_1/2] \cr
1-   \sqrt 2 ( (n_1-j ) / n_1 )^{3/2} , & j \in [n_1 /2 : n_1] \cr 
 \end{cases} \ \ . 
\end{eqnarray}
Finally, if there is no sign change: $q_{t_i}q_{t_{i-1}}=1$, we let
\begin{equation}\label{nochange2.equation} 
q_{t_{i-1} +j} q_{t_{i-1}}:= 1-  \biggl ( {4 j ( n_i -j ) \over  n_i n_{\rm max} } \biggr )^{3/2} .  
\end{equation}
Thus when $\lambda$  satisfies (\ref{lambdak.equation}) for an appropriate constant $c_2$,
then for a constant $C_2$ the bound (\ref{gammabound.equation}) is true for the effective
sparsity.

\subsubsection{Interpolating vector and effective sparsity for $k=3$} \label{effective-sparsity3.section}
For the noisy case, we invoke a scaled and discrete variant of
\begin{equation}\label{q3.equation}
{\rm q}(x): = \begin{cases}  +1- {4 \over 19 } (4x)^{5 \over 2}, & 0 \le x \le {1 \over 4} \cr 
-{5  \over 38} 4^3 ({1 \over 2} - x )^3 + { 35  \over 38} 4 ({1 \over 2} - x ) ,& {1 \over 4 } \le x \le {3 \over 4 } \cr
-1+ {4  \over 19} 4^{5/2} (1- x)^{5 \over 2} ,& { 3 \over 4} \le x \le 1 \cr 
\end{cases} \ .  
\end{equation} 
Note that  ${\rm q}$ is decreasing, ${\rm q}$, ${\rm q}^{\prime}$ and ${\rm q}^{\prime \prime} $ are continuous, and 
$$ {\rm q}^{\prime \prime \prime} (x) = \begin{cases}-{15 \times 16 \over 19}  {1 \over \sqrt x}, & 0 < x < {1 \over 4} \cr
{ 30 \times 32 \over 19},  & {1\over 4}  \le  x <  {3 \over 4} \cr -{ 15 \times 16  \over 19}  {1 \over \sqrt {1-x}}  ,&
{3 \over 4}  \le x < 1\cr 
 \end{cases}  \ . $$
 
 If $n_i / 4 \in {\mathbb N}$ the  rescaled and discrete variant when $q_{t_i}=-1$, $ q_{t_{i-1}} =1$ is
 \begin{equation}\label{change3.equation} 
 { q}_{t_i+j} :=  \begin{cases} 1- \bar a_0 { (4j/n_i)^{5/2} } , & 1 \le j \le n_i/4 \cr
 -  \bar a_3 4^3 { ((n_i/2  - j )/n_i )^3 }+  \bar a_1 4{ (n_i/2 -j) / n_i  } , & n_i / 4 \le j \le 3 n_i/4  \cr
 -1+  \bar a_0 { 4^{5/2} ((n_i -j)/ n_i)^{5/2}  } , & 3n_i/4 \le i \le n_i \cr 
   \end{cases} \ \ . 
\end{equation}
where $\bar a_0$, $\bar a_1$ and $\bar a_3$ can be calculated using
the following lemma with $d= n_i / 4$.
(In the notation of Subsection \ref{effective-sparsityk.section} $a_0 =  4^{3/2} \bar a_0 $,
$a_1= 4 \bar a_1 $ and $a_3= 4^3 \bar a_3 $. )
\begin{lemma} \label{q3.lemma} Let $d \in {\mathbb N}$ and define
\begin{eqnarray*}
 \alpha_1&:=&{[\Delta (d+1)^{5/2} ]  \over d^{3/2} }
 =  { (d+1)^{5/2} - d^{5/2} \over d^{3/2} },  \\
   \gamma_1&:= &{[\Delta d^3] \over d^2} := { d^3 - (d-1)^3 \over d^2 } ,\\
    \alpha_2 &:=& { [\Delta (2)(d+2)^{5/2} ]\over d^{1/2}} :={  [\Delta (d+2)^{5/2} ]  - [\Delta (d+1)^{5/2} ]   \over d^{1/2} },\\
\gamma_2 &:=& { [\Delta(2)  d^3 ]\over d }:=  { [\Delta d^3 ] - [ \Delta (d-1)^3 ] \over d } . 
\end{eqnarray*}
Let
\begin{eqnarray*} 
\bar a_0 &: = & {\gamma_2\over \gamma_2 - \alpha_2 + (\gamma_1 \alpha_2+ \alpha_1 \gamma_2 ) }, \\
 \bar a_3 &:= & { \alpha_2  \over \gamma_2 - \alpha_2 + (\gamma_1 \alpha_2 + \alpha_1 \gamma_2 ) } ,\\
 \bar a_1 &:=& {\gamma_1 \alpha_2+ \alpha_1 \gamma_2  \over \gamma_2 - \alpha_2 + (\gamma_1 \alpha_2+ \alpha_1 \gamma_2 ) }  ,
 \end{eqnarray*}
and for $j\in \{ d, d+1 , d+2 \}$
\begin{eqnarray*}
{\bf q}_j &:=& 1- \bar a_0 {j^{5/2} / d^{5/2} } ,\\
{\bf p}_j& := & -\bar a_3 {(2d - j )^3 / d^3}+ \bar a_1 { (2d-j) / d } . 
\end{eqnarray*}
Then
$$ \Delta (l) { \bf q}_{d+l} = \Delta(l)  {\bf  p}_{d+l}, \ l \in \{ 0,1,2\} . $$
\end{lemma}

\begin{proof}[Proof of Lemma \ref{q3.lemma}]
See Appendix \ref{sparsityproofs.section}.
\end{proof}

The values of the parameters $\bar a_0$, $\bar a_1$ and $\bar a_3$ in the above  lemma depend on $d$, but one easily checks that 
for $d \rightarrow \infty$:
$\alpha_1 \approx 5/2$, $\gamma_1 \approx 3$, 
$\alpha_2 \approx 15/4$ and $\gamma_2 \approx 6$.
Hence $\bar a_0 \approx {4 \over 19} $, $\bar a_3 \approx 5/ 38 $ and $\bar a_1 \approx 35/38 $ as in (\ref{q3.equation}).
If $n_i / 4 \notin {\mathbb N} $ we have similar calculations: the discrete derivatives are then
to match  at - say - $\lfloor n_i / 4  \rfloor$ and $\lceil 3 n_i / 4 \rceil$.
(By the same arguments as for $k=2$ one may without loss of generality assume
that $n_i$ is even.) 

For the boundary intervals we have similar expressions and
when $q_{t_{i-1}}q_{t_{i} }= 1$ we take $q_j$, $j \in [t_{i-1}: t_i ] $,  as for the general $k$ case.
This gives when $\lambda$  satisfies (\ref{lambdak.equation}) for some appropriate constant $c_3$,
then for a constant $C_3$ the bound (\ref{gammabound.equation}) for the effective sparsity.

\subsubsection{Interpolating vector and effective sparsity for $k=4$}\label{effective-sparsity4.section}
For $k=4$ and $i \in S^{\pm}$ we take a scaled and discrete version of the function
${\rm q} : [0 ,1/2] \rightarrow [0, 1]$ defined
(up to rounding errors) as,
$${\rm q} (x):= \begin{cases} 1- (18.62) x^{7/2}  ,& 0 \le x \le {1 \over 6} \cr 
 (44.34) x^4  -(46.19) x^3+ (10.16)x^2 -(1.10)  x +1.05, & {1 \over 6}  \le x \le {1 \over 3}  \cr 
-  (12.93 )(1/2-x)^3 + (4.23)(1/2-x) , & {1 \over 3}  \le x \le {1 \over 2}  \cr   \end{cases} \ . $$

%
%
%
\begin{figure}
\centering 
\includegraphics[width=2.5in,height=2.5in]{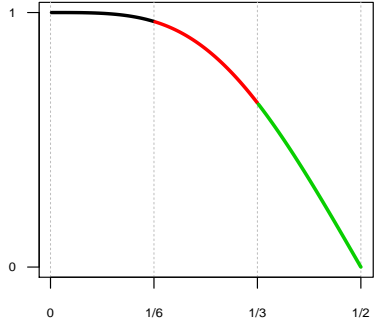}
\caption{The function ${\rm q} : [0, 1/2] \rightarrow [0,1]$ for $k=4$.} 
\end{figure}


 The function ${\rm q}$ is decreasing 
 with ${\rm q}$, ${\rm q}^{\prime}$, ${\rm q}^{\prime \prime}$ ${\rm q}^{\prime \prime \prime} $
 continuous. It was calculated by solving 8 equations with 8 unknowns, following  the description
 in Subsection \ref{effective-sparsityk.section}.
We can now reason as in Subsection \ref{effective-sparsity3.section} to obtain  for
$\lambda$ satisfying  (\ref{lambda1.equation}) the bound (\ref{gammabound.equation}) for the effective
sparsity where $c_k$ and $C_k$ are appropriate constants.

%
%
%

\subsection{Proof of Theorem \ref{dualthm101} } \label{proofmain.section}
Taking $\bar \N_{-S}$ as  the direct product of $\N_{-S}$ and an appropriate space
spanned by $(k-1)s$ additional variables, we derived 
in Subsection \ref{dictionaryk.section} a bound for the length of the columns
of the dictionary $\Psi^{-S}$. With these, we saw that the requirement (\ref{lambda.equation}) for $\lambda$ can
is true when (\ref{lambda1.equation}) holds. Then in Subsections \ref{effective-sparsity1.section}, 
\ref{effective-sparsity2.section}, \ref{effective-sparsity3.section} and \ref{effective-sparsity4.section} we derived a bound for 
the effective sparsity $\Gamma ( q_S, w_{-S})$
when (\ref{lambda1.equation}) is strengthened to (\ref{lambdak.equation}). 
Theorem \ref{dualthm101} thus follows from the adaptive bound  in Theorem \ref{dualthm45}
for the general analysis problem. 
\hfill $\sqcup \mkern -12mu \sqcap$

\section{Conclusion}\label{dualsec12}


The sharp oracle inequalities with fast rates show that for $k \in \{1 ,2,3,4 \}$ the estimator adapts to the unknown number of jumps in the $(k-1)^{\text{th}}$ discrete derivative and provide finite-sample prediction bounds. In particular, these show that the prediction error of the total variation regularized estimator is upper bounded by the optimal trade-off between approximation error and estimation error.  The key tool for providing these results is bounding the
effective sparsity using interpolating vectors. This approach allows extension to other problems
as well, for instance higher dimensional extensions. See  \cite{Ortelli2d} where the Hardy-Krause variation
serves as regularizer. For total variation on graphs, one may apply the fact that the dictionary can
be formed by counting the number of times an edge is used when traveling from a given
node to all other nodes. This can then be done on the sub-graphs formed by removing the active edges.
For graphs with cycles there are several paths from one node to another. One may then
choose those that allow for a smooth interpolating vector. Finally, the approach can be extended to estimation problems with
$\ell_1$-penalty on the discrete derivative but 
loss functions other than least squares (see  \cite{vdG2020} for the case of logistic loss
with total variation penalty on the canonical parameter).

{\bf Acknowledgements.}  We acknowledge support for this project
from the the Swiss National Science Foundation (SNF grant 200020\_169011). We thank the Associate Editor and two Referees for their very helpful remarks.

\bibliographystyle{plainnat}

\bibliography{library}

\appendix

\section{Proof of Theorem 2.2} \label{Thm2.2.proof}

Let $\e := Y - f_0$ be the noise vector. 
The proof of oracle inequalities starts from the following basic inequality.
\begin{lemma}[Basic inequality]\label{duallem31}
For all $f \in \R^n$ it holds that
$$ \norm{\hf-f^0}^2_n + \norm{\hf-f}^2_n \le \norm{f-f^0}^2_n + \frac{2 \e'(\hf-f)}{n} + 2 \lambda (\norm{Df}_1- \norm{D\hf}_1).$$
\end{lemma}

\begin{proof}
See Lemma B.1 in \cite{orte19-2}.
\end{proof}

To derive  oracle results from the basic inequality, we have to control the increments of the empirical process given by ${ \e'(\hf-f)}/{n}$. 
Inspired by \cite{dala17}, we  decompose the increments of the empirical process into a part projected onto $\bar \N_{-S}$ and a remainder.
For $f \in \R^n$  we let $ w_{-S} (Df)_{-S} := \{ w_j( Df)_j \}_{j \notin \D \backslash S } $.

\begin{lemma}[Bound on the empirical process with mock variables.]\label{duallem38}
Choose $\lambda$ such that (\ref{lambda.equation}) holds and let $w_{-S}$ be as in Definition \ref{dualdef41}.
  Let $v>0$. It holds that with probability at least $1-e^{-u}-e^{-v}$,
  $$\frac{\e'f}{n}\le \left( \sqrt{\frac{\bar r_S}{n}}+ \sqrt{\frac{2v}{n}}  \right)\norm{f}_n + \lambda \norm{ w_{-S}
(Df)_{-S}}_1 , \ \forall f \in \R^n.$$
\end{lemma}

\begin{proof}[Proof of Lemma \ref{duallem38}]
We decompose the empirical process as
$$ \e'f/{n}=\e' f_{\bar \N_{-S}} /{n} + {\e' f_{\bar \N_{-S}^{\perp} } }/{n}.$$

\begin{itemize}
\item For $v>0$ define the set 
$$ \mathcal{V }:= \left\{ \norm{\e_{\bar \N_{-S}}}_n \le  \sqrt{\frac{\bar r_S}{n}}+ \sqrt{\frac{2v}{n}} 
 \right\}.$$
 By applying  Lemma 1 in \cite{laur00} (Lemma 8.6 in \cite{vand16}, a concentration inequality for $\chi^2$ random variables) to $\mathcal{V}$ we get that $\pr(\mathcal{V})\ge 1-e^{-v}$.

On $\mathcal{V}$ we have that
$$ \frac{\e' f_{\bar \N_{-S} } }{n}\le \norm{\e_{\bar \N_{-S}}}_n \norm{\bar f_{\N_{-S} } }_n  \le  \left( \sqrt{\frac{\bar r_S}{n}}+
\sqrt{\frac{2v}{n}}  \right) \norm{f}_n.$$

\item For $\lambda_0(u)=
\sqrt {(2 \log (2(m-s)) + 2u) /n } $ defined  as in (\ref{lambda0.equation})
define the set $$\mathcal{U}:= \left\{ \frac{\abs{\epsilon'{\psi}_j^{-S} }/n }{\norm{{\psi}_j^{-S}}_n}\le \lambda_0 (u), 
\ \forall \ j \in \D\setminus \tilde{S} \right\}.$$
We apply  a standard concentration inequality for the maximum of $m-s$ standard Gaussian random variables
(this can be found e.g.\ in Lemma 17.5 in \cite{vand16}) to deduce that
$ \pr(\mathcal{U})\ge 1-e^{-u}$. 

Since $f_{\bar \N_{-S}^{\perp}} = \Psi^{-S} (D f)_{-S} $ by the definition of the dictionary $\Psi^{-S}$,
it holds on 
$\mathcal{U}$ that
$${\e' f_{\bar \N_{-S}^{\perp} } }/{n}
\le \lambda \norm{{w}_{-{S}}(Df)_{-S}}_1$$
where
$$ w_j := \| \psi_j^{-S} \|_n \lambda_0 (u) / \lambda , \ j \in \D \backslash S .$$
 \end{itemize}
The proof of the lemma is completed by noting that $\pr(\mathcal{U}\cap \mathcal{V})\ge 1-e^{-u}-e^{-v}$.
\end{proof}

\begin{lemma}[A bound using $\Gamma^2(q_S , w_{-S}) $.]\label{duallem43}
Let $f\in \R^n $ and $S \subset \D $ be arbitrary, and $q_S := \text{sign}(Df)_S$. Then  for all $\tilde f \in \R^n $
$$ \| (Df )_S \|_1 - \| (D\tilde f )_S \|_1 - \| (1-w_{-S} )( D( \tilde f - f))_{-S} \|_1  \le \Gamma (q_S , w_{-S}) \| f - \tilde f \|_n  . $$
\end{lemma}

\begin{proof} Clearly $\| (Df)_S \|_1  = q_S^{\prime} (Df)_S$. Moreover $\| ( D\tilde f )_S\|_1 \ge q_S ( D \tilde f)_S $.
Thus
$$ \| (D f)_{S} \|_1 - \| (D \tilde f )_{S} \|_1 \le q_S^{\prime} (D( f - \tilde f))_{S} .$$
The result now follows from the Definition \ref{dualdef41} of effective sparsity.
\end{proof}

\begin{proof}[Proof of Theorem \ref{dualthm45}]
By combining Lemma \ref{duallem31} and Lemma \ref{duallem38} we get that with probability at least $1-e^{-u}-e^{-v}$
\begin{eqnarray*}
 \norm{\hf-f^0}^2_n + \norm{\hf-f}^2_n & \le & \norm{f-f^0}^2_n + 
 2 \left(  \sqrt{\frac{\bar r_S}{n}} + \sqrt{\frac{2v}{n}} \right)\norm{\hf-f }_n \\ &+ & 2 \lambda \norm{w_{-S} D (\hf - f )_{-S} 
}_1 +   2 \lambda (\norm{Df}_1- \norm{D\hf}_1).
 \end{eqnarray*}
 
The theorem we are proving has a non-adaptive bound and an adaptive one.\\
$\bullet$ For establishing the non-adaptive bound, we 
note that $\norm{w_{-S} }_{\infty}\le 1$, so that 
\begin{eqnarray*}
\norm{w_{-S} (D (\hf - f ))_{-S} 
}_1 +     \norm{Df}_1- \norm{D\hf}_1 \le 2 \| Df \|_1 
\end{eqnarray*}
 and thus with probability at least $1-e^{-u}-e^{-v}$
\begin{eqnarray*}
 \norm{\hf-f^0}^2_n + \norm{\hf-f}^2_n & \le & \norm{f-f^0}^2_n +  
2 \left(  \sqrt{\frac{\bar r_S}{n}}+ \sqrt{\frac{2v}{n}}\right)\norm{\hf-f }_n \\ &+ & 
   4 \lambda \norm{Df}_1.
 \end{eqnarray*}
 The non-adaptive bound therefore follows from the conjugate inequality
 $2ab \le a^2 + b^2 $, $a,b \in \R$.\\
 $\bullet$ For the adaptive bound we apply the inequalities
 \begin{eqnarray*}
& & \norm{w_{-S} (D (\hf - f ))_{-S} 
}_1 +     \norm{Df}_1- \norm{D\hf}_1\\ &  \le & \| (1+ w_{-S})  (Df)_{-S} \|_1 + \| (Df)_S \|_1  
- \| (Df)_S \|_1 - \norm{(1- w_{-S} )(D\hf) _{-S} }_1 \\
& \le &2  \|   (Df)_{-S} \|_1 + \| (Df)_S \|_1  
- \| (D\hf)_S \|_1 - \norm{(1- w_{-S} )(D(\hf- f ) ) _{-S} }_1.
\end{eqnarray*}
Thus with probability at least $1-e^{-u}-e^{-v}$
\begin{eqnarray*}
& &  \norm{\hf-f^0}^2_n + \norm{\hf-f}^2_n \\ &\le& \norm{f-f^0}^2_n +  
 2 \left( \sqrt{\frac{\bar r_S}{n}} +\sqrt{\frac{2v}{n}} +\right)\norm{\hf-f }_n +
 4 \lambda   \|   (Df)_{-S} \|_1\\ &  + & 2\lambda \biggl (  \| (Df)_S \|_1  
- \| (D\hf)_S \|_1 - \norm{(1- w_{-S} )(D(\hf- f ) ) _{-S} }_1\biggr ) \\
 &\le &  \norm{f-f^0}^2_n +  
2  \left(  \sqrt{\frac{\bar r_S}{n}} +\sqrt{\frac{2v}{n}} \right)\norm{\hf-f }_n + 
 4 \lambda   \|   (Df)_{-S} \|_1\\ &  + & 2\lambda \Gamma (q_S , w_{-S} )\| \hf - f \|_n  
  \end{eqnarray*}
  where in the last inequality we used Lemma \ref{duallem43}.
 The proof is completed by applying again the conjugate inequality
  $2ab \le a^2 + b^2 $, $a,b \in \R$.
\end{proof}

\section{Proofs for Section 3.1} \label{dictionaryproofs.section}

\subsection{Proof for the result in Subsection \ref{dictionaryk.section}}

\begin{lemma}[Symmetry of $\norm{{\psi}^{\D}_j}_2$.]\label{symmetry.lemma}
Let $k \in [1: n-1]$ and $\Psi^{\cal D} = \Delta (k)^+$.
For all $j \in \D$ we have that
$$ \norm{\psi_j^{\D}}^2_2=\norm{\psi_{n+k+1-j}^{\D}}^2_2 .$$
\end{lemma}

\begin{proof}[Proof of Lemma \ref{symmetry.lemma}]
$$\text{For }r=\{r_i\}_{i=1}^n= \begin{pmatrix}r_1\\ \vdots \\ r_n\end{pmatrix}\in \R^n \text{ define }  \circlearrowright r= \{r_i\}_{i=n}^1= \begin{pmatrix}r_n\\ \vdots \\ r_1\end{pmatrix}.$$
Note that, for $r,\tilde{r} \in \R^n$ we have that $r'\tilde r= (\circlearrowright r) ' \circlearrowright \tilde r$.

Since $\Delta(k)$ is of full rank, we have that ${\Delta(k)}^+= {\Delta(k)}'(\Delta(k){\Delta(k)}')^{-1}.$

Let $r_i \in \R^{n-k}$ be the $i^{\text{th}}$ row of ${\Delta(k)}'$, i.e. ${\Delta(k)}'=\{r_i\}_{i=1}^n $. We observe that:
\begin{itemize}
\item For $k$ even, $r_i= \circlearrowright r_{n+1-i},\ i \in [n]$;
\item For $k$ odd, $r_i= - \circlearrowright r_{n+1-i},\ i \in [n]$.
\end{itemize}

Define $P:=(\Delta(k){\Delta(k)}')^{-1}\in \R^{(n-k)\times (n-k)}$ and let the rows and columns of $P$ be indexed by the set $\D=[n]\setminus [k]$. We have that $P$ is symmetric and all its diagonal entries are the same. Therefore, if we denote by $P_j$ the $j^{\text{th}}$ column of $P$, we note that $ P_j= \circlearrowright P_{n+k+1-j}, \ j \in \D$.

We distinguish two cases:
\begin{itemize}
\item When $k$ is even
\begin{eqnarray*}
\psi_{n+k+1-j}^{\D}&=&\{r_i P_{n+k+1-j}\}_{i=1}^n= \{r_i \circlearrowright P_j \}_{i=1}^n= \{\circlearrowright r_{n+1-i} \circlearrowright P_j\}_{i=1}^n\\
&=& \{ r_{n+1-i} P_j\}_{i=1}^n= \{r_iP_j\}_{i=n}^1 = \circlearrowright \{r_iP_j\}_{i=1}^n= \circlearrowright {\psi_j^\D}. 
\end{eqnarray*}

\item When $k$ is odd, by similar calculations $ \psi_{n+k+1-j}^{\D} = - \circlearrowright \psi_j^{\D}$.
\end{itemize}

Since, for $r \in \R^n$, $\norm{r}^2_2= \norm{\pm \circlearrowright r}^2_2$ we get the claim.
\end{proof}

To calculate the pseudo inverse of $\Delta(k)$ we proceed as follows (cf. Lemma 2.2 in \cite{orte19-2}).
\begin{enumerate}
\item We select the matrix $A(k) \in \R^{k \times n}$, s.t.
$$ A(k)_{ij}= \begin{dcases} (-1)^{j+i} \binom{i}{j}, & j=i-l, \ l \in \{0, \ldots, i-1\},\ i \in [k],\\
0, & \text{else}. \end{dcases}$$
\item We find $X^k= \begin{pmatrix} A(k)\\ \Delta(k) \end{pmatrix}^{-1} \in \R^{n \times n}$. We have that $X^k= \{\phi^k_j\}_{j \in [n]}$, where
\begin{itemize}
\item for $k=1$, $\phi^1_j= 1_{\{i \ge j\}}, \ i,\ j \in [n]$,
\item and for $k \ge 2$, 
$$ \phi_j^k= \begin{dcases} \phi^j_j, & 1 \le j <k\\
\sum_{l \ge j} \phi_l^{k-1}, & k \le j \le n.\end{dcases}$$
This is the falling factorial basis for equidistant design, see \cite{wang2014falling} and \cite{Ryan2020}].
In the notation for the general analysis problem (see the beginning of
Subsection \ref{dictionary.section})  $\{ \psi_j \}_{j=1}^n :=
\{ \phi_j^k \}_{j=1}^n$ is a complete dictionary.
\end{itemize}
\item Then $\Delta^+(k)= \{ \psi_j^{\D} \}_{j \in [k+1 : n]}$, where
$$ \psi_j^{\D}  = ( {\phi}^k_j)_{\N_{\D}^{\perp}} , j \in [k+1 :  n].$$
\end{enumerate}

\begin{proof}[Proof of Lemma \ref{duallem39}]
We  roughly estimate
$$ \norm{\psi_j^{\D}}_2^2 \le \norm{{\phi_j}^k}_2^2, \ ({n+k+1})/{2} \le 
j \le n ,$$
where
$$ \norm{{\phi_j}^k}^2_2 \le \int_{0}^{n+1-j} x^{2k-2} dx \le (n+1-j)^{2k-1}.$$
By symmetry (Lemma \ref{symmetry.lemma}), we  obtain the claim.
\end{proof}

\subsection{Proofs for the results in Subsection \ref{dictionary2.section}}
  
  \begin{proof}[Proof of Lemma \ref{duallem51}]
Let $\psi_1:\equiv 1\in \R^n$ and 
$$\psi_j: =\{(i-j+1)1_{\{i \ge j\}}\}_{ i \in [1:n]} , \ j \in [2:n] . $$
We want to find the anti-projections of the vectors $\psi_j,\ j\in [3:n]$ onto the linear space spanned by $\psi_1$ and $\psi_2$.

We use the Gram-Schmidt procedure to orthonormalize the basis on which we want to project.

By $u_1, u_2$ we denote two vectors orthogonal to each other, which span the linear span of $\psi_1,\ \psi_2$, and by $e_1, e_2$ their normalized version.
We take $u_1=\psi_1\equiv 1$. Then $ {e_1= n^{-1/2}}$.
We now take $u_2= \psi_2- \langle \psi_2, e_1 \rangle e_1$. We have that $\langle \psi_2, e_1 \rangle= \frac{n (n-1)}{2 n^{1/2}}$ and thus
$ u_2= \left\{(i-1)1_{\{i \ge 2\}}-\frac{n-1}{2} \right\}_{i=1}^n$. The norm of $u_2$ is $\norm{u_2}^2_2= \frac{(n+1)n (n-1)}{12}$ and it follows that
$$e_2=\sqrt{\frac{12}{(n+1)n (n-1)}} \left\{(i-1)1_{\{i \ge 2\}}-\frac{n-1}{2} \right\}_{i=1}^n .$$

Let $\bar{\psi}_j$ denote the projection of $\psi_j$ onto the linear span of $e_1, \ e_2$ and let $\psi_j^{\D}= \psi_j-\bar{\psi}_j$ be denote the anti-projection. It holds that
$$ \bar\psi_j= \langle \psi_j, e_1 \rangle e_1 + \langle \psi_j, e_2 \rangle e_2 \text{ and } \norm{\bar \psi _j}^2_2= \langle \psi_j, e_1 \rangle^2  + \langle \psi_j, e_2 \rangle^2.$$
Moreover by Pythagoras $ \norm{\psi_j^{\D}}^2_2= \norm{\psi_j}^2_2- \norm{\bar \psi_j}^2_2$.

To compute the length of the anti-projections we thus have to compute the coefficients of the projections onto the orthonormal vectors spanning the linear space we project onto (i.e. $\langle \psi_j, e_1 \rangle$ and $\langle \psi_j, e_2 \rangle$) and the lengths of the vectors to project (i.e. $\norm{\psi_j}^2_2$).

We omit all the steps of the computations, which were performed with the support of the software ``Wolfram Mathematica 11''. We present directly the results, that for the inner products $\langle \psi_j, e_1 \rangle$ and $\langle \psi_j, e_2 \rangle$ are

$$ \langle \psi_j, e_1 \rangle = \frac{(n-j+1)(n-j+2)}{2 \sqrt{n}},$$
$$ \langle \psi_j, e_2 \rangle = \sqrt{\frac{1}{12(n+1)n (n-1)}}(n-j+1)(n-j+2) (n+2j-3).$$
The length of the vectors to project is given by
$$ \norm{\psi_j}_2^2= \frac{(n-j+1)(n-j+2)(2n-2j+3)}{6}.$$
For the length of the projections we obtain the expression
$$ \norm{\bar{\psi_j}}^2_2= \frac{(n-j+1)^2(n-j+2)^2}{4n}\left(1+ \frac{(n+2j-3)^2}{3(n-1)(n+1)} \right).$$
For the length of the anti-projections we obtain the exact expression
$$ \norm{\psi_j^{\D}}^2_2= \frac{(n-j+1)(n-j+2)(j-2)(j-1)(2j(n-j+3)-3(n+1))}{6n(n+1)(n-1)}.$$
\end{proof}


\subsection{Proof for the result in Subsection \ref{dictionary3.section}}

\begin{proof}[Proof of Lemma \ref{duallem61}]
Let $\psi_1:\equiv 1$, $\psi_2:=\left\{ i-1 \right\}_{i \in [1:n]}$, 
$$\psi_j:=\left\{(i-j+1)(i-j+2)1_{\{i \ge j\}}/2\right\}_{i\in [1:n]}, \  j \in [2:n]. $$
The length of the anti-projections is given by
$$ \norm{\psi_j^{\D}}^2_2 = \norm{\psi_j}^2_2- \langle \psi_j, e_1 \rangle^2-\langle \psi_j, e_2 \rangle^2-\langle \psi_j, e_3 \rangle^2.$$

The orthonormal basis vectors $e_1$ and $e_2$ are the same as in the proof of Lemma \ref{duallem51}. Here as well the computations have been dome with the support of the software ``Wolfram Mathematica 11''. 
In a first step we want to find
$$u_3= \psi_3- \langle \psi_3, e_1 \rangle e_1 - \langle \psi_3, e_2 \rangle e_2$$
and its normalized version $e_3= u_3/\norm{u_3}_2$.

We use the Gram-Schmidt process.
%
We have that
\begin{eqnarray*}
&&\norm{\psi_j}^2_2= \sum_{i=1}^n 1_{\{i \ge j\}} \frac{(i-j+1)^2(i-j+2)^2}{4}\\
&=& \frac{(n+3-j)(n+2-j)(n+1-j)(10-12j+3j^2+12n-6jn+3n^2)}{60}.
\end{eqnarray*}
Moreover, for the coefficients of the projections onto 
$e_1$ 
 we have
$$ \langle \psi_j, e_1 \rangle= \frac{(n+3-j)(n+2-j)(n+1-j)}{6\sqrt{n}}.$$

For the coefficients of the projections onto 
$e_2$ 
we have that
$$ \langle \psi_j, e_2 \rangle= \frac{(n+3-j)(n+2-j)(n+1-j)(n+j-2)}{\sqrt{48(n+1)n(n-1)}}.$$
%
We thus obtain that the anti-projection of $\psi_3$ onto $\text{span}(\psi_1, \psi_2)$ is given by
$$ u_3= \psi_3-\langle \psi_3, e_1 \rangle e_1-\langle \psi_3, e_2 \rangle e_2= \left\{ \frac{(i-1)(i-n)}{2}+ \frac{(n-1)(n-2)}{12} \right\}_{i=1}^n.$$
The $\ell_2$-norm of $u_3$ is
$$ \norm{u_3}^2_2= \frac{(n+2)(n+1)n(n-1)(n-2)}{720}$$
and the third vector $e_3$ of the orthonormal basis writes as
\begin{eqnarray*}&&e_3= u_3/ \norm{u_3}_2\\
&=&\sqrt{\frac{720}{(n^2-4)(n^2-1)n}} \left\{\frac{(i-1)(i-n)}{2}+ \frac{(n-1)(n-2)}{12} \right\}_{i=1}^n.
\end{eqnarray*}

We can now compute the  coefficient of the projections of $\psi_j$ onto $e_3$:
$$ \langle \psi_j, e_3 \rangle = $$ $$ \frac{(n+3-j)(n+2-j)(n+1-j)(6j^2+3jn-24j+n^2-6n+20)}{\sqrt{720(n+2)(n+1)n(n-1)(n-2)}}.$$

 Combining the formulas for the quantities we found, we  get the claim.
\end{proof}

\section{Proofs for Section 3.3}\label{sparsityproofs.section} 
 
 \subsection{Proof for the result in Subsection \ref{effective-sparsityk.section}}

\begin{proof}[Proof of Lemma \ref{sqrtj.lemma}]
We have for $j \ge k$
\begin{eqnarray*}
 \Delta (k) j^{2k-1\over 2} &=& \sum_{l=0}^k { k \choose l} (-1)^l (j-l)^{2k-1 \over 2} \\
& = &j^{2k-1 \over 2} \biggl [ \sum_{l=0}^k {k \choose l} (-1)^l \biggl (1- { l \over j} \biggr )^{2k-1\over 2} \biggr ] . 
\end{eqnarray*} 
We do a $(k-1)$-term Taylor expansion of $x \mapsto (1- x)^{2k-1 \over 2} $ around $x=0$: 
$$ 
(1- x)^{2k-1 \over 2} =\sum_{i=0}^{k-1} a_i x^i + {\rm rem} (x) ,$$
where $a_0=1$, $a_1= -{2k-1 \over 2} , \cdots$ are the coefficients of the
Taylor expansion and where the remainder ${\rm rem} (x)$ satisfies for some constant ${\rm C}_k$
$$ \sup_{0 \le x \le 1/2 } | {\rm rem} (x) |\le {\rm C}_k | x |^k . $$
Thus
 \begin{eqnarray*}
 & &    \sum_{l=0}^k { k \choose l} (-1)^l\biggl (1- { l \over j} \biggr )^{2k-1\over 2} \\
 & =& \sum_{l=0}^k { k \choose l} (-1)^l \sum_{{\rm k}=0}^{k-1} a_i \biggl ({ l\over j } \biggr )^i + \sum_{l=0}^k { k \choose l} (-1)^l {\rm rem} \biggl ( {l \over j} \biggr ) \\
 & = & \underbrace{\Delta (k)  {\bf p} }_{=0} +  \sum_{l=0}^k  { k \choose l} (-1)^l  {\rm rem}\biggl ( {l \over j} \biggr ) ,
\end{eqnarray*} 
where
$$ {\bf p} = \biggl \{ (-1)^k \sum_{i=0}^{k-1} { a_i  } \biggl ( { l\over j } \biggr )^i \biggr  \}_{l=0, \ldots , k-1  }$$
is a polynomial of degree $k-1$ and hence $\Delta (k) {\rm p} =0 $.
 It follows that for $j \ge 2k$, 
$$ \biggl | \sum_{l=0}^k { k\choose l} (-1)^l\biggl (1- { l \over j} \biggr )^{2k-1\over 2}\biggr | \le 
 \sum_{l=0}^k  { k \choose l} \biggl | {\rm rem} \biggl ( {l \over j} \biggr ) \biggr |    \le \tilde {\rm C}_k {1 \over j^k } .$$
But then for $j \ge 2k$
 $$ \Delta (k)  j^{2k-1 \over 2}  \le \tilde {\rm C}_k {1 \over \sqrt j } . $$
 So
 $$ \sum_{j=2k}^d | \Delta (k)  j^{2k -1 \over 2}  |^2 \le \tilde {\rm C}_k^2 (1 +\log d) . $$
 Finally, for $k \le j < 2k$,
 $$ \Delta (k)  j^{2k-1 \over 2} \le   j^{2k-1 \over 2 } \sum_{l=0}^k { k \choose l}  
  \le 2^k k^{2k-1 \over 2} .$$
 Thus
 $$\sum_{j=k}^d | \Delta (k) j^{2k -1 \over 2}  |^2 \le (2k)^{2k} + \tilde {\rm C}_k^2 (1+  \log  d )\le
 \tilde C_k^2 (1+ \log d) $$
 for some constant $\tilde C_k $.
 \end{proof}
 
 \subsection{Proof of the result in Subsection \ref{effective-sparsity3.section}}

\begin{proof}[Proof of Lemma \ref{q3.lemma}]
First
\begin{eqnarray*}
 \Delta(2) {\bf q}_{d+2}& = &-  {\bar a_0 [\Delta(2) (d+2)^{5/2} ] \over d^{5/2}}  = - { \bar a_0 \alpha_2  \over d^2}\\
 & = & -{1 \over d^2} {\alpha_2 \gamma_2 \over \gamma_2 - \alpha_2 + (\gamma_1 \alpha_2+ \alpha_1 \gamma_2 ) } \\
\Delta(2) {\bf p}_{d+2} &=& -{ \bar a_3 [\Delta (2) d^3 ]  \over d^3}  = -{ \bar a_3 \gamma_2 \over d^2} =  - {1 \over d^2} { \alpha_2 \gamma_2  \over \gamma_2 - \alpha_2 + (\gamma_1 \alpha_2- \alpha_1 \gamma_2 ) }\  . 
\end{eqnarray*}
Second
\begin{eqnarray*}
 \Delta {\bf q}_{d+1} &= &- { \bar a_0 [\Delta (d+1)^{5/2} ] \over d^{5/2} } =
-{ \bar a_0 \alpha_1 \over d} = - {1 \over d} {\gamma_2 \alpha_1 \over \gamma_2 - \alpha_2 + (\gamma_1 \alpha_2+ \alpha_1 \gamma_2 ) }\\
 \Delta {\bf p}_{d+1} &=&  {\bar a_3 [\Delta d^3 ] \over d^3 } - {\bar a_1 \over d} 
= { \bar a_3 \gamma_1 \over d } - {\bar a_1 \over d} \\
& =&{1 \over d}  { \alpha_2 \gamma_1 -  (\gamma_1 \alpha_2+ \alpha_1 \gamma_2)  \over \gamma_2 - \alpha_2 + (\gamma_1 \alpha_2 + \alpha_1 \gamma_2 ) } =
-{1 \over d} { \alpha_1 \gamma_2  \over \gamma_2 - \alpha_2 + (\gamma_1 \alpha_2 + \alpha_1 \gamma_2 ) }\ .
\end{eqnarray*} 
Finally
\begin{eqnarray*}
 {\bf q}_d &=& 1- {\bar a_0 } = 1-  {\gamma_2\over \gamma_2 - \alpha_2 + (\gamma_1 \alpha_2+ \alpha_1 \gamma_2 ) }  \\
&=&
{ \gamma_1 \alpha_2+ \alpha_1 \gamma_2  -\alpha_2  \over \gamma_2 - \alpha_2 + (\gamma_1 \alpha_2 + \alpha_1 \gamma_2 ) }\\
{\bf  p}_d &= & - \bar a_3 +\bar a_1 =  { \gamma_1 \alpha_2+ \alpha_1 \gamma_2  -\alpha_2  \over \gamma_2 - \alpha_2 + (\gamma_1 \alpha_2 + \alpha_1 \gamma_2 ) }\ .
\end{eqnarray*}
\end{proof}

\end{document}